\documentclass[suppldata]{interact}

\usepackage{epstopdf,amsmath, amsfonts, amssymb,centernot, tikz, cite}% To incorporate .eps illustrations using PDFLaTeX, etc.
\usepackage[caption=false]{subfig}% Support for small, `sub' figures and tables

\usepackage[numbers,sort&compress]{natbib}% Citation support using natbib.sty
\bibpunct[, ]{[}{]}{,}{n}{,}{,}% Citation support using natbib.sty
\theoremstyle{plain}% Theorem-like structures provided by amsthm.sty
\newtheorem{theorem}{Theorem}[section]
\newtheorem{lemma}[theorem]{Lemma}

\newtheorem{proposition}[theorem]{Proposition}

\theoremstyle{definition}
\newtheorem{definition}[theorem]{Definition}

\theoremstyle{remark}
\newtheorem{remark}{Remark}

\newcommand{\pvec}[1]{\vec{#1}\mkern2mu\vphantom{#1}}
\DeclareMathOperator{\Tr}{Tr}
\DeclareMathOperator{\deter}{det}
\DeclareMathOperator{\re}{Re}
\DeclareMathOperator{\im}{Im}

\begin{document}

\title{Embedded eigenvalues for perturbed periodic Jacobi operators using a geometric approach}

\author{
\name{E. Judge\textsuperscript{a}\thanks{CONTACT E. J. Author. Email: ej75@kent.ac.uk}, S. Naboko\textsuperscript{b}, I. Wood\textsuperscript{a}\thanks{CONTACT I. W. Author. Email: i.wood@kent.ac.uk}}
\affil{\textsuperscript{a}School of Mathematics, Statistics and Actuarial Science, Sibson Building, University of Kent, Canterbury, Kent, CT2 7FS, UK; \textsuperscript{b}Dept. Math. Physics, Institute of Physics, St. Petersburg State University, 1 Ulianovskaia, St. Petergoff,
St. Petersburg, 198504, Russia.}}

\maketitle

\begin{abstract}
We consider the problem of embedding eigenvalues into the essential spectrum of periodic Jacobi operators, using an oscillating, decreasing potential. To do this we employ a geometric method, previously used to embed eigenvalues into the essential spectrum of the discrete Schr\"{o}dinger operator. For periodic Jacobi operators we relax the rational dependence conditions on the values of the quasi-momenta from this previous work. We then explore conditions that permit not just the existence of infinitely many subordinate solutions to the formal spectral equation but also the embedding of infinitely many eigenvalues.
\end{abstract}

\begin{keywords}
Jacobi matrices; periodic operators; embedded eigenvalues; spectral theory; Wigner-von Neumann.
\end{keywords}

\section{Introduction.}

Embedded eigenvalues for Schr\"{o}dinger operators can be traced all the way back to von Neumann and Wigner~\cite{7}, where oscillating potentials via diffractive interference were used to produce the desired bound states. At the time these calculations were considered mere oddities with few physical applications, however, due to the Schr\"{o}dinger equation's key role in quantum mechanics, it has since been suggested that these bound states might be found in particular molecular and atomic systems \cite{8,9,10,11,12,13}. Indeed, physical evidence has actually been recorded to support these assertions in semi-conductor heterostructures~\cite{14}. For Jacobi matrices, which are the discrete analogue of these problems, the subject of eigenvalues has also been well explored (see, for example, \cite{gol1, gol2} for complex potentials and \cite{9a,9f,9j,9g,2,3,4} for real potentials).

Here we continue this exploration and extend the techniques in \cite{9g} from their original discrete Schr\"{o}dinger operator (DSO) setting to arbitrary period-$T$ Jacobi operators, that is operators, $J_T$, that are tri-diagonal and of the form
\begin{equation}\label{4.10}
J_T:=\left(\begin{array}{ccccccccccccc}
 b_1&a_1&\\
 a_1&b_2&a_2&\\
 &a_2&b_3&a_3&\\
 &&\ddots&\ddots&\ddots&\\
 &&&a_{T-1}&b_T&a_T\\
 &&&&a_T&b_1&a_1&\\
 &&&&&a_1&b_2&a_2&\\
 &&&&&&\ddots&\ddots&\ddots&\\
 &&&&&&&a_{T-1}&b_{T}&a_T\\
 &&&&&&&&a_T&b_1&a_1&\\
 &&&&&&&&&a_1&b_2&a_{2}&\\
 &&&&&&&&&&\ddots&\ddots&\ddots
 \end{array}\right),\end{equation} with $a_i,b_i\in\mathbb{R}, a_i>0$ for all $i$. Throughout, we deal exclusively with an oscillating, diagonal perturbation, $q_k$, which is slowly decreasing, i.e. a Wigner-von Neumann type potential. In \cite{21} we used the the Wigner-von Neumann technique to construct an explicit Wigner-von Neumann potential starting from an ansatz for the eigenvector. It allowed us to embed a single eigenvalue for an arbitrary periodic Jacobi operator. Similarly, in \cite{6} we used the Levinson-type asymptotic method \cite{3.16,7.6,janmosz} to construct an explicit potential, with various frequencies, to produce subordinate solutions for infinitely many spectral parameters of the formal spectral equation of a perturbed period-$T$ Jacobi operator. In our present paper we use a completely different approach to describe an explicit geometric procedure for the construction of the potential and present not just subordinate solutions but eigenvectors for (possibly) infinitely many eigenvalues of a perturbed period-$T$ Jacobi operator.

The structure of the paper is as follows. First, we expound the geometric approach to embed a single eigenvalue into the essential spectrum of a periodic Jacobi operator (Section~\ref{sec2}). The technique in \cite{9g} is only valid for the special case of the DSO and those $\lambda$ in the generalised interior of the essential spectrum rationally independent with $\pi$; however, the technique constructed in this paper works for an arbitrary period-$T$ Jacobi operator and for any element in the generalised interior of the essential spectrum with quasi-momentum not equal to $\frac{\pi}{2}$. Adapting the idea of \cite{9g} we then construct a method to  embed infinitely many eigenvalues, simultaneously, into the essential spectrum, providing their quasi-momenta are rationally independent with each other and $\pi$ (Section~\ref{sec1}). These ideas are then developed, further, to permit the embedding of infinitely many eigenvalues, simultaneously, with fewer constraints. As examples we consider the cases when one eigenvalue, $\lambda$, has quasi-momentum rationally dependent with $\pi$ (Section~\ref{sec3}); when two eigenvalues, $\lambda_1, \lambda_2$, have quasi-momenta rationally dependent with $\pi$ (Section~\ref{sec3}); and also arbitrarily (but finitely) many, $\lambda_i$, with quasi-momentum rationally dependent with $\pi$, providing certain co-prime conditions are satisfied by the denominators of the quasi-momenta (Section~\ref{sec5}). Some of our results on embedding multiple eigenvalues are new even for the case of the DSO.

\section{Preparatory case. Embedding a single eigenvalue.}\label{sec2}

Let $J_T$ be a periodic Jacobi operator and $B_i(\lambda)$ be the {\it transfer matrices} where $$B_i(\lambda):=\left(\begin{array}{cc} 0&1\\ -\frac{a_{i-1}}{a_i}&\frac{\lambda-b_i}{a_i}\end{array}\right), i\in\{1,\dots,T\}, a_0:=a_T, \lambda\in\mathbb{C}.$$ Then $M(\lambda):=B_T(\lambda)\dots B_1(\lambda)$ is the associated {\it monodromy matrix}. We define the {\it hyperbolic points} to be those $\lambda$ that produce a monodromy matrix with two real eigenvalues, $\mu_1,\mu_2$ where
$|\mu_1|>1$ and $|\mu_2|<1$. The {\it elliptic points} are those $\lambda$ that produce a monodromy matrix with two distinct complex eigenvalues, $\mu,\overline{\mu}$ of modulus one. The elliptic points are characterised by the condition that $\left|\Tr M(\lambda)\right|<2$. The {\it parabolic points} are those $\lambda$ that produce a monodromy matrix with one eigenvalue, equal to either $1$ or $-1$, with algebraic multiplicity $2$. For the relevance of this partition of the complex plane to the spectral theory of periodic Schr\"{o}dinger and Jacobi operators, see \cite{19,20}.  Furthermore, we define the {\it generalised interior of the essential spectrum}, $\sigma_{ell}(J_T)$, to be the set of elliptic points, i.e. $$\sigma_{ell}(J_T)=\{\lambda\in\mathbb{R}~|~\left|\Tr M(\lambda)\right|<2\}.$$ For $\lambda\in\sigma_{ell}(J_T)$, the eigenvalues $\mu_{\pm}$ of the monodromy matrix, $M(\lambda)$, are such that $\mu_{\pm}=e^{\pm i\theta(\lambda)}$, where the function $\theta(\lambda)$ is called the {\it quasi-momentum}, and has range $(0,\pi)$.

We state now that throughout this paper the potential $(q_k)$ is always real and converges to zero. Since it is sufficient for our construction, we assume $q_k=0$ unless $k\equiv 1 \mod T$, and $\widetilde{q}_i$ is defined to be the $i$-th non-zero entry of the sequence $(q_k)$. We introduce the perturbed monodromy matrix, $M_i(\lambda)$, such that
\begin{align*}
M_i(\lambda)&:=B_T(\lambda)B_{T-1}(\lambda)\dots B_1(\lambda-\widetilde{q}_i)\\
&=B_T(\lambda)\dots B_2(\lambda)\left(B_1(\lambda)-\frac{\widetilde{q}_i}{a_1}\left(\begin{array}{cc}
0&0\\
0&1\end{array}\right)\right)\\
&=M(\lambda)-\frac{\widetilde{q}_i}{a_1}A(\lambda),
\end{align*} where $A(\lambda):=M(\lambda)B_1^{-1}(\lambda)\left(\begin{array}{cc}
0&0\\
0&1\end{array}\right).$ For $\lambda\in\sigma_{ell}(J_T)$, we can write $$M_i(\lambda)=W(\lambda)\left[R(\theta(\lambda))-\frac{\widetilde{q}_i}{a_1}W^{-1}(\lambda)A(\lambda)W(\lambda)\right]W^{-1}(\lambda),$$ where $W(\lambda):=V(\lambda)U(\lambda),$ the columns of $V(\lambda)$ are composed of the eigenvectors for $\mu(\lambda)$ and $\overline{\mu(\lambda)}$ respectively, $U(\lambda)$ is some unitary matrix and $R(\theta)$ is a rotation matrix of angle $\theta(\lambda)$.

The following lemma will be needed in the technique:
\begin{lemma}\label{5.9}
The invertible matrix $W(\lambda)$ can always be chosen to be real.
\end{lemma}

\begin{proof}
The matrix $W(\lambda)=V(\lambda)U(\lambda)$, where both $U(\lambda)$ and $V(\lambda)$ are invertible, is also invertible. However, the matrix $W(\lambda)$ may not always be real, but we do have a choice in what $U(\lambda),V(\lambda)$ we use to define it. All that is needed is that $W(\lambda)$ satisfies the relation \begin{equation}\label{5.7}M(\lambda)W(\lambda)=W(\lambda)R(\theta(\lambda)).\end{equation} Indeed, there are many possibilities for a real $W(\lambda)$; for instance, we could just take the $\re(W(\lambda))$ or $\im(W(\lambda))$ of any $W(\lambda)$ that satisfies \eqref{5.7}, however, this does not guarantee the invertibility condition. So, assume for contradiction that $\re(W(\lambda))+\beta \im(W(\lambda))$ is not invertible for any $\beta\in\mathbb{R}$. Then, define the function $$f(\beta):=\deter\left[\re(W(\lambda))+\beta \im(W(\lambda))\right],$$ and by our assumption $f(\beta)=0$  for all $\beta\in\mathbb{R}$. But the function $f$ is just a polynomial in $\beta$ and is therefore analytic in $\beta$, and can be extended analytically to the complex plane. Thus, it is zero everywhere on the plane. However, if we choose $\beta=i$ then we get the original expression for $W(\lambda)$, which is always invertible, and thus we have a contradiction. This means that there exists at least one $\beta\in\mathbb{R}$ such that the matrix $\re(V(\lambda)U(\lambda))+\beta \im(V(\lambda)U(\lambda))$ is both invertible and real. We choose this matrix to be $W(\lambda)$.~\qedhere
\end{proof}

Throughout this paper $\lambda\in\sigma_{ell}(J_T)$ and therefore $M(\lambda),B_j(\lambda)$ are real as $\lambda$ is real. Moreover, by Lemma~\ref{5.9} we may assume from now on that $W(\lambda)$ is real. Consequently, since all the matrices have real entries and $\widetilde{q}_i$ is real and decreasing, we have for $\vec{f}\in\mathbb{R}^2$ $$\|\left(R(\theta(\lambda))-\frac{\widetilde{q}_i}{a_1}W^{-1}(\lambda)A(\lambda)W(\lambda)\right)\vec{f}\|^2$$
\begin{align}
&=\bigg\langle \left(R^*(\theta(\lambda))-\frac{\widetilde{q}_i}{a_1}W^*(\lambda)A^*(\lambda)\left(W^{-1}(\lambda)\right)^*\right)\nonumber\\
&~~~~~~~~~~~~~~~~~~~~~~~~\times\left(R(\theta(\lambda))-\frac{\widetilde{q}_i}{a_1}W^{-1}(\lambda)A(\lambda)W(\lambda)\right)\vec{f},\vec{f}\bigg\rangle\nonumber\\
&=\bigg\langle\left (I-\frac{2\widetilde{q}_i}{a_1} \re\left(R^*(\theta(\lambda))W^{-1}(\lambda)A(\lambda)W(\lambda)\right)+{{O}}(\widetilde{q}_i^2)\right)\vec{f},\vec{f}\bigg\rangle\nonumber\\
&=\|\vec{f}\|^2+O(\widetilde{q}_i^2)\|\vec{f}\|^2-\frac{2\widetilde{q}_i}{a_1}\langle R(-\theta(\lambda))W^{-1}(\lambda)A(\lambda)W(\lambda)\vec{f},\vec{f}\rangle,\label{5.11}
\end{align} and where the second equality was obtained using that for an arbitrary real matrix $B$, $\re (B)=\frac{1}{2}(B+B^*)$. Additionally,  observing that the matrix $A(\lambda)$ is rank one yields:
\begin{align}
&\langle R(-\theta(\lambda))W^{-1}(\lambda)A(\lambda)W(\lambda)\vec{f},\vec{f}\rangle=\langle A(\lambda)W(\lambda)\vec{f}, (W^{-1}(\lambda))^*R(-\theta(\lambda))^*\vec{f}\rangle\nonumber\\
&=\langle M(\lambda)B_1^{-1}(\lambda)\left(\begin{array}{cc}
0&0\\
0&1\end{array}\right)W(\lambda)\vec{f},\left(W^{-1}(\lambda)\right)^*R^*(-\theta(\lambda))\vec{f}\rangle\nonumber\\
&=\langle M(\lambda)B_1^{-1}(\lambda)\langle W(\lambda)\vec{f},\vec{e}_2\rangle\vec{e}_2,(W^{-1}(\lambda))^*R^*(-\theta(\lambda))\vec{f}\rangle\nonumber\\
&=\langle W(\lambda)\vec{f},\vec{e}_2\rangle\langle M(\lambda)B_1^{-1}(\lambda)\vec{e}_2,(W^{-1}(\lambda))^*R^*(-\theta(\lambda))\vec{f}\rangle\nonumber\\
&=\langle \vec{f},W^*(\lambda)\vec{e}_2\rangle \langle R(-\theta(\lambda))W^{-1}(\lambda)M(\lambda)B_1^{-1}(\lambda)\vec{e}_2,\vec{f}\rangle\nonumber\\
&=\|\vec{f}\|^2\|W^*(\lambda)\vec{e}_2\|\left\|W^{-1}(\lambda)M(\lambda)B_1^{-1}(\lambda)\vec{e}_2\right\|\nonumber\\
&~~\times\cos(\vec{f},W^*(\lambda)\vec{e}_2)\cos(R(-\theta(\lambda))W^{-1}(\lambda)M(\lambda)B_1^{-1}(\lambda)\vec{e}_2,\vec{f}),\label{5.12}
\end{align}
where $(x,y)$ means the angle between $x$ and $y$, and $\vec{e}_1=(1,0)^T$, $\vec{e}_2=(0,1)^T$. Now by combining Equations~\eqref{5.11} and \eqref{5.12} we obtain
\begin{multline}\label{5.8}
\|\left(R(\theta(\lambda))-\frac{\widetilde{q}_i}{a_1}W^{-1}(\lambda) A(\lambda)W(\lambda)\right)\vec{f}\|^2\\=\|\vec{f}\|^2\bigg(1-2\frac{\widetilde{q}_i}{a_1}\|W^*(\lambda)\vec{e}_2\|\left\|W^{-1}(\lambda)M(\lambda)B_1^{-1}(\lambda)\vec{e}_2\right\|\\ \times\cos(\vec{f},W^*(\lambda)\vec{e}_2)\cos(R(-\theta(\lambda))W^{-1}(\lambda)M(\lambda)B_1^{-1}(\lambda)\vec{e}_2,\vec{f})+O(\widetilde{q}_i^2)\bigg).
\end{multline}

The next stage of the argument is to establish for what $\vec{f}$ we can reduce the magnitude of the expression on the righthand-side of Equation~\eqref{5.8}. In order to do this we must investigate the relationship between the vectors $R(-\theta(\lambda))W^{-1}(\lambda)M(\lambda)B_1^{-1}(\lambda)\vec{e}_2$ and $W^*(\lambda)\vec{e}_2$.

\begin{lemma}\label{5.26}
Let $\lambda\in\sigma_{ell}(J_T)$, where $J_T$ is an arbitrary period-$T$ Jacobi operator and $\theta(\lambda)$, abbreviated to $\theta$, is the quasi-momentum. Then the vectors $W^*(\lambda)\vec{e}_2$ and $R(-\theta)W^{-1}(\lambda)M(\lambda)B_1^{-1}(\lambda)\vec{e}_2$ are always orthogonal with respect to the standard complex inner product.
\end{lemma}

\begin{proof}
By observing that $$M^{-1}(\lambda)=(W(\lambda)R(\theta)W^{-1}(\lambda)))^{-1}=W(\lambda)R(-\theta)W^{-1}(\lambda),$$ the result clearly follows from the explicit calculation:
$$\langle W^*(\lambda)\vec{e}_2,R(-\theta)W^{-1}(\lambda)M(\lambda)B_1^{-1}(\lambda)\vec{e}_2\rangle$$
 \begin{align*}
 &=\langle \vec{e}_2,W(\lambda)R(-\theta)W^{-1}(\lambda)M(\lambda)B_1^{-1}(\lambda)\vec{e}_2\rangle\\
 &=\langle \vec{e}_2,M^{-1}(\lambda)M(\lambda)B_1^{-1}(\lambda)\vec{e}_2\rangle\\
 &=\langle \vec{e}_2,B_1^{-1}(\lambda)\vec{e}_2\rangle=\langle \vec{e}_2,-\frac{a_1}{a_T}\vec{e}_1\rangle= 0.~\qedhere
 \end{align*}
\end{proof}

To estimate the cosines appearing in \eqref{5.8} we use the following elementary lemma.

\begin{lemma}\label{5.6}
Consider $\vec{h}_0,\vec{h}_1\neq {\underline{0}}\in\mathbb{R}^2$. Then for arbitrarily small $0<\epsilon<\frac{\pi}{2}$ we have $$|\sin(\vec{h}_0,\vec{f})\sin(\vec{h}_1,\vec{f})|\geq \sin^2(\epsilon)$$ for $\vec{f}\in\mathbb{R}^2$ lying outside the four cones with central axis $\pm\vec{h}_0,\pm\vec{h}_1$ and opening angles less than or equal to $\epsilon.$
\end{lemma}

\begin{proof}
Using that $(\vec{h}_0,\vec{f})>\epsilon,(\vec{h}_1,\vec{f})>\epsilon$ and that $\sin(x)$ is a monotonic increasing function on the range $(0,\frac{\pi}{2})$ yields the result.~\qedhere
\end{proof}

Consequently, by defining $\vec{h}_0:=\vec{v}_0$ where $\vec{v}_0\bot W^*(\lambda)\vec{e}_2$, and $\vec{h}_1:=\vec{v}_1$ where
\newline $\vec{v}_1\bot R(-\theta)W^{-1}M(\lambda)B_1^{-1}(\lambda)\vec{e}_2$ then
\begin{align*}
|\cos(W^*(\lambda)\vec{e}_2,\vec{f})\cos(R(-\theta)W^{-1}(\lambda)M(\lambda)B_1^{-1}(\lambda)\vec{e}_2,\vec{f})|&=|\sin(\vec{h}_0,\vec{f})\sin(\vec{h}_1\vec{f})|\\
&\geq \sin^2(\epsilon)>0
\end{align*}
for all $\vec{f}$ lying outside the four cones with central axis $\pm\vec{h}_0,\pm\vec{h}_1$ and opening angles less than or equal to $\epsilon.$

Then, for $\vec{f}$ outside a collection of arbitrarily narrow cones around $\pm W^*(\lambda)\vec{e}_2,\pm R(-\theta)W^{-1}(\lambda)M(\lambda)B_1^{-1}(\lambda)\vec{e}_2$ of opening angle $\epsilon$ (see Figure \ref{1.8} for more details), we have for any vector $\vec{f}$ the estimate  \begin{equation}\label{1.9}\|\left(R(\theta)-\frac{\widetilde{q}_i}{a_1}W^{-1}(\lambda)A(\lambda)W(\lambda)\right)\vec{f}\|^2\leq\|\vec{f}\|^2\left(1-C(\lambda)|\widetilde{q}_i|+O\left(\widetilde{q}_i^2\right)\right)
  \end{equation}
for some $C(\lambda)>0$, providing that the sign of $\widetilde{q}_i$ is chosen appropriately. (Otherwise the corresponding components of the solution will not decrease in size, but increase.)

We will use the following definition throughout this section.

\begin{definition}
We define the set $S_\epsilon$  to be the four orthogonal cones about the vectors $\pm W^*(\lambda)\vec{e}_2$ and $\pm R(-\theta)W^{-1}(\lambda)M(\lambda)B_1^{-1}(\lambda)\vec{e}_2$ each with opening angle $\epsilon$ (see Figure~\ref{1.8} for more details).
\end{definition}

\begin{figure}
\centering
\begin{tikzpicture}
        \draw [thick, red , -] (0,-2) -- (0,4);
        \draw [dashed, red, -] (0,0) -- (4,1);
        \draw [dashed, red, -] (0,0) -- (4,-1);
        \draw [thick, black] (2,1/2) to [out=-45, in=45] (2,0)
        node [left,black] at (2,1/4) {$\epsilon$};
         \draw [dashed, red, -] (0,0) -- (4,-1);
        \draw [thick, black] (2,-1/2) to [out=45, in=-45] (2,0)
        node [left,black] at (2,-1/4) {$\epsilon$};
           \draw [ thick, red, -] (-4,0) -- (4,0);
        \draw [dashed, red, -] (0,0) -- (-1,4);
     \draw [dashed, red, -] (0,0) -- (1,4);
      \draw [thick, black] (-1/2,2) to [out=45, in=135] (0,2)
        node [below,black] at (-1/4,2) {$\epsilon$};
         \draw [thick, black] (1/2,2) to [out=135, in=45] (0,2)
        node [below,black] at (1/4,2) {$\epsilon$};
         \draw [thick, black] (-2,1/2) to [out=-135, in=135] (-2,0)
        node [right,black] at (-2,1/4) {$\epsilon$};
           \draw [thick, black] (-2,0) to [out=-135, in=135] (-2,-1/2)
        node [right,black] at (-2,-1/4) {$\epsilon$};
        \draw [dashed, red, -] (0,0) -- (-1/2,-2);
     \draw [dashed, red, -] (0,0) -- (1/2,-2);
        \draw [thick, black] (-1/2,-2) to [out=-45, in=-135] (0,-2)
        node [above,black] at (-1/4,-2) {$\epsilon$};
         \draw [thick, black] (1/2,-2) to [out=-135, in=-45] (0,-2)
        node [above,black] at (1/4,-2) {$\epsilon$};

      \draw [dashed, red, -] (0,0) -- (-4,1);
     \draw [dashed, red, -] (0,0) -- (-4,-1);

\end{tikzpicture}
\caption{The four solid lines represent the central axes of the four bad cones comprising $S_\epsilon$ whilst the dashed lines represent the width of the respective cones, each with opening angle $\epsilon$. }
\label{1.8}
\end{figure}
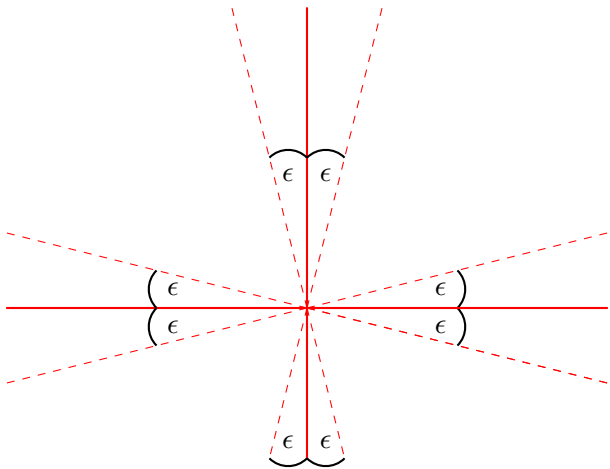

Thus, it has been established that for any $\lambda\in\sigma_{ell}(J_T)$ and for any $\vec{f}\in\mathbb{R}^2\setminus S_\epsilon$ we have the relations \begin{align}\left\|M_i(\lambda)\vec{f}~\right\|^2&=\left\|W(\lambda)\left(R(\theta)-\frac{\widetilde{q}_i}{a_1}W^{-1}(\lambda)A(\lambda)W(\lambda)\right)W^{-1}(\lambda)\right\|^2\nonumber\\
&\leq \|W(\lambda)\|^2\|W^{-1}(\lambda)\|^2\|\vec{f}\|^2\left(1- C(\lambda)|\widetilde{q}_i|+O\left(\widetilde{q}_i^2\right)\right)\label{5.65}\end{align} for some constant $C(\lambda)>0$, i.e. providing we ignore the contribution from $W(\lambda),W^{-1}(\lambda)$ (which will be dealt with later) then for a small enough potential the vector $\vec{f}$ shrinks. However it still remains to show that any arbitrary initial vector can be moved into the region $\mathbb{R}^2\setminus S_\epsilon$, to undergo shrinkage, regardless of where in the plane it is initially and the value of the quasi-momentum, $\theta(\lambda)$. Lemma~\ref{5.5} will establish this, however the following proposition is needed first.

\begin{proposition}\label{5.14}{\rm{\cite{5.3,5.2}}} Let $\vec{e}_1,\vec{e}_2,\dots,\vec{e}_N$ and $\pvec{e}_1',\pvec{e}_2',\dots,\pvec{e}_N'$ be two arbitrary collections of unit vectors in $\mathbb{R}^2$, and $\{\pi,\theta_1,\dots,\theta_N\}$ a collection of rationally independent numbers. If $R(\theta)$ is the operator of rotation through an angle $\theta$ around the origin of coordinates in $\mathbb{R}^2$, then for any $\gamma>0$ there is a number $m\in\mathbb{N}$ such that the angle between the vectors $\pvec{e}_i'$ and $R^m(\theta_i)\vec{e}_i$ is smaller than $\gamma$ simultaneously for all $i=1,2,\dots,N$. Here $m\leq R(N,\gamma)$, where the constant $R$ depends only on $N$ and $\gamma$.
  \end{proposition}

Consequently, for $\lambda\in\sigma_{ell}(J_T)$ such that $\theta(\lambda)\not\in{\pi}\mathbb{Q}$, Proposition~\ref{5.14} gives an upper bound on the number of applications of the rotation matrix,  $R(\theta(\lambda))$, required to rotate an arbitrary vector, $\vec{f}\in\mathbb{R}^2$, into an acceptable region of the plane, $\mathbb{R}^2\setminus S_\epsilon$.

The following gives an analogous result for when $\theta(\lambda)$ is not rationally independent with $\pi$, although we do have to exclude from our consideration the case when $\theta(\lambda)=\frac{\pi}{2}$.

\begin{lemma}\label{5.5}
Let $\lambda\in\sigma_{ell}(J_T)$ with $\theta(\lambda)=\frac{p\pi}{q}\neq\frac{\pi}{2}$ where $\gcd(p,q)=1$.  Then, for any real vector, $\vec{f}\in\mathbb{R}^2$ and any $\epsilon\in\left(0,\frac{\pi}{2q}\right)$ there exists some $k\in\{0,1,2\}$ such that $R(k\theta)\vec{f}\not\in S_{\epsilon}$.
\end{lemma}

\begin{proof}
If $\frac{\theta(\lambda)}{\pi}=\frac{p}{q}$, where $\gcd(p,q)=1$, then there are two cases:\newline
(Case One)~If $\vec{f}\not\in S_{\epsilon}$ then take $k=0$.\newline
(Case Two)~Otherwise $\vec{f}\in S_{\epsilon}$ and it needs to be shown that after a certain number of rotations the vector $\vec{f}$ no longer inhabits any of the four orthogonal bad cones. Define $\xi$ to be the angle between the vector $\vec{f}$ and the central axis of the cone containing $\vec{f}$ (see Figure~\ref{fig55} for more details). Clearly, we have the relation
 \begin{equation}\label{5.1}
|\xi|<\epsilon.
 \end{equation}
  Moreover, since $2\epsilon<\frac{\pi}{q}<\theta(\lambda)$, we also have that one application of the rotation matrix, $R(\theta)$, ensures that the vector $\vec{f}$ is moved to outside the first cone.  We now establish that the vector $\vec{f}$ is not moved into the bad cone opposite under the same single application of $R(\theta)$, i.e. $|\pi-\left(\xi+\frac{p\pi}{q}\right)|\geq\epsilon$. Recalling that $\epsilon<\frac{\pi}{2q}$ and \eqref{5.1}, we assume for contradiction
 \begin{align}
 \left|\pi-\left(\xi+\frac{p\pi}{q}\right)\right|<\epsilon &\Rightarrow\left|\pi-\frac{p\pi}{q}\right|<2\epsilon\nonumber\\
 &\iff \left|q\pi-p\pi\right|<\frac{2\pi}{2q}\cdot q\nonumber\\
 &\Rightarrow \left|q-p\right|\pi<\pi.\label{5.25}
  \end{align}
   This implies $q=p$, which is a contradiction. Thus the condition $\epsilon<\frac{\pi}{2q}$ guarantees that a rotation will move the vector, $\vec{f}$, outside of the first cone and not into the `third', i.e. the mirror image of the first.

 It still remains to consider
  \begin{equation}\label{5.2}
  \left|\frac{\pi}{2}-\left(\xi +\frac{p\pi}{q}\right)\right|<\epsilon,
  \end{equation}
  which describes the instance of applying the rotation matrix, $R(\theta)$, once, but only succeeding to move the vector from the first bad cone into the (orthogonal) `second'. Clearly, if this doesn't happen then one rotation is enough to move $\vec{f}$ outside of $S_\epsilon$ as $\theta<\pi$. Thus, we assume this is not the case and that \eqref{5.2} happens. We then consider the case of applying another rotation and moving from the second cone into the third. This is described by \begin{equation}\label{5.3} \left|{\pi}-\left(\xi +\frac{2p\pi}{q}\right)\right|<\epsilon.\end{equation}
   Then, by recalling that $\epsilon<\frac{\pi}{2q}$ and combining Equations \eqref{5.1} and \eqref{5.3} we see that
  \begin{align*}
  \left|{\pi}-\left(\xi +\frac{2p\pi}{q}\right)\right|<\epsilon&\Rightarrow \left|\pi-\frac{2p\pi}{q}\right|<2\epsilon\\
  &\iff \left|q\pi-{2p\pi}\right|<\frac{\pi}{2q}\cdot2 q\\
  &\Rightarrow \left|{q}-2p\right|\pi<\pi.
  \end{align*} This implies $q=2p$ which is a contradiction. This tells us that \eqref{5.3} can never happen; in particular, after at most two rotations the vector $\vec{f}$ will be rotated out of the set $S_\epsilon$.~\qedhere
\end{proof}

\begin{figure}
\centering
\begin{tikzpicture}
        \draw [thick, red , -] (0,-2) -- (0,4);
        \draw [dashed, red, -] (0,0) -- (4,1);
        \draw [dashed, red, -] (0,0) -- (4,-1);
        \draw [thick, black] (2,1/2) to [out=-45, in=45] (2,0)
        node [left,black] at (2,1/4) {$\epsilon$};
        \draw [thick, blue, ->] (0,0) -- (4,-1/2)
        node [right, blue] at (4,-1/2) {$\vec{f}$};
    \draw [ thick, red, -] (-4,0) -- (4,0);
    \draw [thick, black] (3,0) to [out=-45, in=45] (3,-3/8)
    node [left, black] at (3,-3/16) {$\xi$};
    \draw [dashed, red, -] (0,0) -- (-1,4);
     \draw [dashed, red, -] (0,0) -- (1,4);
      \draw [thick, black] (-1/2,2) to [out=45, in=135] (0,2)
        node [below,black] at (-1/4,2) {$\epsilon$};
         \draw [thick, black] (1/2,2) to [out=135, in=45] (0,2)
        node [below,black] at (1/4,2) {$\epsilon$};
         \draw [thick, black] (-2,1/2) to [out=-135, in=135] (-2,0)
        node [right,black] at (-2,1/4) {$\epsilon$};
           \draw [thick, black] (-2,0) to [out=-135, in=135] (-2,-1/2)
        node [right,black] at (-2,-1/4) {$\epsilon$};
        \draw [dashed, red, -] (0,0) -- (-1/2,-2);
     \draw [dashed, red, -] (0,0) -- (1/2,-2);
        \draw [thick, black] (-1/2,-2) to [out=-45, in=-135] (0,-2)
        node [above,black] at (-1/4,-2) {$\epsilon$};
         \draw [thick, black] (1/2,-2) to [out=-135, in=-45] (0,-2)
        node [above,black] at (1/4,-2) {$\epsilon$};

      \draw [dashed, red, -] (0,0) -- (-4,1);
     \draw [dashed, red, -] (0,0) -- (-4,-1);

\end{tikzpicture}
\caption{The four solid lines represent the central axes of the four bad cones comprising $S_\epsilon$ whilst the dashed lines represent the width of the respective cones, each with opening angle $\epsilon$. The vector $\vec{f}$ in this particular example resides in the rightmost bad cone, and is of angle $\xi$ from the central axis of the nearest bad cone.}
\label{fig55}
\end{figure}
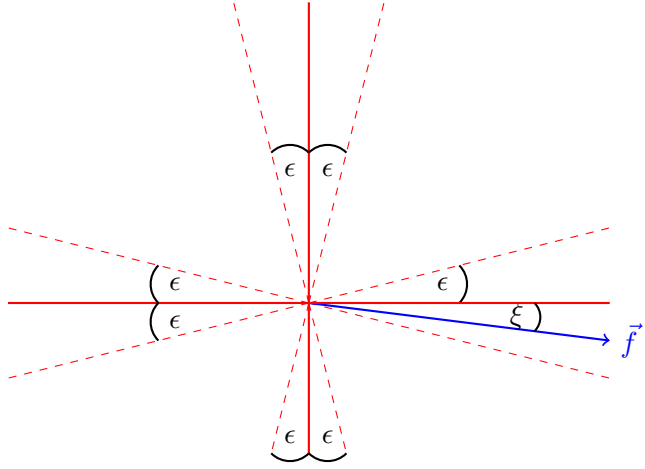

\begin{remark}
The geometric difficulty of those $\lambda$ with $\theta(\lambda)=\frac{\pi}{2}$ arises from the orthogonality of the bad cones. Thus if for these $\lambda$ the vector $\vec{f}$ should fall into one of the bad cones, no amount of rotating will ever relocate the vector into a shrinkable area of the domain (i.e. a region outside the four arbitrarily small bad cones) since the vector will just move from one bad cone into another.
\end{remark}

\begin{remark}
It should be stressed not only for this section, but also the ones that follow, that the only truly bad regions in the plane are along the central axes of the bad cones. Thus, when proving certain results (i.e. Lemmas~\ref{5.5}, \ref{5.43}, \ref{5.62} and \ref{5.63}) on the feasibility of manoeuvring certain vectors into shrinkable regions of the plane we only focus on avoiding the central axis of the bad-cones. This is because regardless of how close the vector $\vec{f}$ is to the central axis, providing the vector is not on it, the opening angle, $\epsilon$, can simply be reduced so that the bad cone is narrow enough for the vector, $\vec{f}$, to avoid the forbidden region altogether. However, this should not be misconstrued as a statement advocating the possibility of bad-lines, rather than bad-cones (arbitrarily small as they may be). The need for some opening angle, $\epsilon$, about the bad cones follows from Lemma~\ref{5.6} which tells us that if $\epsilon$ is zero then it is possible to choose $\vec{f}$ arbitrarily near to the central axis such that $$|\cos(W^*(\lambda)\vec{e}_2,\vec{f})\cos(R(-\theta)W^{-1}(\lambda)M(\lambda)B_1^{-1}(\lambda)\vec{e}_2,\vec{f})|\geq |\sin(\epsilon)|^2= 0.$$ In which case we only obtain the result $$\|\left(R(\theta)-\frac{\widetilde{q}_i}{a_1}W^{-1}(\lambda)A(\lambda)W(\lambda)\right)\vec{f}\|^2\leq\|\vec{f}\|^2+O\left(\widetilde{q}_i^2\right),$$ in particular there is no shrinking of $\|\vec{f}\|^2$ by a factor $C(\lambda)|\widetilde{q}_i|$.
\end{remark}

It follows from Proposition~\ref{5.14} and Lemma~\ref{5.5} that for all pairs $(\theta(\lambda),\vec{f})$, with the condition $\theta(\lambda)\neq \frac{\pi}{2}$, $\exists~n\leq n_0(\lambda;\epsilon)$, for $\epsilon$ small enough, such that $R(n\theta)\vec{f}\not\in S_\epsilon$, and thus by \eqref{1.9} the vector, $\vec{f}$, can be diminished in size by the application of a perturbed monodromy matrix with an appropriate potential $(q_n)$.

 The candidate eigenvector, and solution to the three-term recurrence problem, $u(\lambda)$ is constructed by choosing $u_0=0$ and $u_1=1$ and then applying the unperturbed monodromy matrices $M(\lambda)$ until the initial components are rotated into an acceptable region of the plane, i.e. $\mathbb{R}^2\setminus S_\epsilon$. By Proposition~\ref{5.14} and Lemma~\ref{5.5} this is guaranteed to happen. Then, the perturbed transfer matrix $M(\lambda-\widetilde{q}_n)$ is applied with an appropriate $\widetilde{q}_n$ so that the magnitude of the components of the eigenvector begin to decrease with factor $(1-C|\widetilde{q}_n|) $. After this shrinkage the initial components are either in the acceptable region (in which case we immediately apply another monodromy matrix, $M(\lambda-\widetilde{q}_{n+1})$) or they are again in $S_\epsilon$ and we apply Proposition~\ref{5.14} and Lemma~\ref{5.5}  again to move into a region where we may obtain shrinkage. We continue this process of rotation and shrinkage ad infinitum. Thus the potential will take the form $$(q_n)=(0,\dots,0,\widetilde{q}_1,0,\dots,0,\widetilde{q}_2,0,\dots,\dots,\widetilde{q}_i,0,\dots,0,\widetilde{q}_{i+1},\dots),$$ in particular the sequence is populated by many zero entries. This is to account for when the monodromy matrix is analogous to a rotation and the fact that only one of the $T$ transfer matrices is perturbed. It is important to note that the $\epsilon$-condition about the ‘bad cones’ is fixed for each $\lambda$.

We now wish to show that the candidate eigenvector, $\underline{u}$, is in the sequence space $l^2(\mathbb{N};\mathbb{R})$. This will use the following lemma.

\begin{lemma}
Let $\lambda\in\sigma_{ell}(J_T)$ such that $\theta(\lambda)\neq \frac{\pi}{2}$. Then, for  $\widetilde{q}_i\ll 1$ and $\left(\widetilde{q}_i\right)\in l^2$,  the candidate eigenvector $\underline{u}:=(u_n)_{n\geq 1}$, obtained by the procedure in this section, satisfies the estimate \begin{multline}\label{5.13}
\|\underline{u}(\lambda)\|^2\leq K_1(\lambda)\sum\limits_{i=0}^\infty\left[\prod\limits_{j=0}^i (1-C(\lambda)|\widetilde{q}_j|)\right](r_{i+1}+1)\\+K_2(\lambda)\sum\limits_{t=1}^\infty \prod_{s=1}^{t-1}(1-C(\lambda)|\widetilde{q}_s|)+K_3(\lambda)
\end{multline}
for some $K_1(\lambda), K_2(\lambda), K_3(\lambda)\in\mathbb{R}$, where $\widetilde{q}_0:=0$ and $r_i$ represents the $i$-th interval of rotation, i.e. the number of rotations required to rotate the  components  following the $i-1$-th shrinkage into the acceptable region, $\mathbb{R}^2\setminus S_\epsilon$, so that the perturbed monodromy matrix $M_i(\lambda)$ can be applied.
\end{lemma}

 \begin{remark}
  The first component of the expression comes from calculating the contribution at every $T$-th term, as well as all those interim components that arise from an incomplete unperturbed monodromy matrix, i.e. those $u_{kT+s}$ such that $$\left(\begin{array}{c}
   u_{kT+s}\\
   u_{kT+s+1}
   \end{array}\right)=B_s(\lambda)\dots B_1(\lambda)\left(\begin{array}{c}
   u_{kT}\\
   u_{kT+1}
   \end{array}\right)$$ for some $k\in\mathbb{N}, 1\leq s<T$. The second term estimates those components that arise from an incomplete perturbed monodromy matrix, i.e. those $u_{kT+s}$ such that $$\left(\begin{array}{c}
   u_{kT+s}\\
   u_{kT+s+1}
   \end{array}\right)=B_s(\lambda)\dots B_1(\lambda-\widetilde{q}_t)\left(\begin{array}{c}
   u_{kT}\\
   u_{kT+1}
   \end{array}\right)$$ for some $k, t\in\mathbb{N}, 1\leq s<T$ and uses the fact that $\|B_1(\lambda-\widetilde{q}_t)\|$ is uniformly bounded in $t$. Finally, the third term factors in the $O(\widetilde{q}^2_k)$ contributions.
\end{remark}

Finally, setting $\widetilde{q}_i:=\frac{c_0}{i}$ for some constant $c_0$, and invoking Proposition~\ref{5.14} and Lemma~\ref{5.5} to establish that there exists some $R$ such that $r_i\leq R$ for all $i$ gives:
 \begin{align*}
&\left[\prod_{j=0}^i\left(1-C(\lambda)|\widetilde{q}_j|\right)\right](r_{i+1}+1)\leq (R+1)e^{\sum\limits_{j=1}^i\log(1-C(\lambda)|\widetilde{q}_j|)}\\
&=(R+1)e^{-\sum\limits_{j=1}^iC(\lambda)|\widetilde{q}_j|+{O(\widetilde{q}_j^2)}}\asymp e^{-C(\lambda)\sum\limits_{j=1}^i\frac{c_0}{j}}\sim e^{-C(\lambda)\cdot c_0\ln i}=\frac{1}{i^{C(\lambda)c_0}}
\end{align*} which is in $l^1(\mathbb{N};\mathbb{C})$ as $i\rightarrow\infty$, providing $C(\lambda)c_0>1$. Thus, we need $c_0>\frac{1}{C(\lambda)}$. This concludes the argument that $\underline{u}\in l^2$.

We can summarise the above result in the following theorem.

\begin{theorem}\label{5.41}
Let $J_T$ be an arbitrary period-$T$ Jacobi operator and $\lambda\in\sigma_{ell}(J_T)\setminus\{x|\theta(x)=\frac{\pi}{2}\}$. Then there exists a potential $q_n=O\left(\frac{1}{n}\right)$, with $q_n\neq o\left(\frac{1}{n}\right)$, and a non-zero vector $(u_n)_{n\geq1}\in l^2(\mathbb{N};\mathbb{R})$ such that $$(J_T+Q)(u_n)_{n\geq 1}=\lambda (u_n)_{n\geq 1}$$ where $Q$ is an infinite diagonal matrix with entries $(q_n)$.
\end{theorem}

\begin{remark}
The Coulomb-type decay of $(q_n)$ gives that the potential is a compact perturbation and therefore the essential spectrum of the operators $J_T$ and $J_T+Q$ coincide. Consequently, the eigenvalue $\lambda$ in Theorem~\ref{5.41} is embedded in the essential spectrum of the operator $J_T+Q$.
\end{remark}

\section{Embedding infinitely many eigenvalues.}\label{sec1}

In this section we adapt the method from \cite{9g} to  embed infinitely many eigenvalues, simultaneously, into the essential spectrum using a single potential, $q_n$. Indeed, the problem that arises whenever one contemplates the embedding of multiple eigenvalues is that not only must all the two dimensional vectors, $$\vec{f}_1,\vec{f}_2,\dots,\vec{f}_n,\dots,$$ corresponding to the respective candidate eigenvalues $$\lambda_1,\lambda_2,\dots,\lambda_n,\dots,$$ inhabit shrinkable areas of the plane, but also that the respective $\vec{f}_i$ are in compatible shrinkable areas, since the reduction factor $$(1-C(\lambda)|\widetilde{q}_i|)\|\vec{f}\|^2$$ used in the single eigenvalue case depended on choosing the sign of the potential-component, $\widetilde{q}_k$, correctly. Otherwise as the components of one candidate eigenvector are decreased, the components for the other eigenvector will possibly be increased. Consequently, in this and the following three sections, it is necessary to impose restrictions on the set of eigenvalues being embedded to ensure it is possible that the initial components, $\vec{f}_i$, can all be manoeuvred simultaneously into compatible, shrinkable regions of the real plane regardless of where they first lie. Throughout this section, we insist that any finite selection of eigenvalues from the infinite set have quasi-momenta rationally independent with each other and $\pi$.

The following definition will be needed throughout the rest of this paper.

\begin{definition}
Let $S_\epsilon^{(i)}$ denote the four arbitrarily small orthogonal cones each with opening angle $\epsilon$ about the vectors $\pm R(-\theta(\lambda_i))W^{-1}(\lambda_i)M(\lambda_i)B_1^{-1}(\lambda_i)\vec{e}_2$ and $\pm W^*(\lambda_i)\vec{e}_2$.
\end{definition}

Without loss of generality, the orthogonal bad cones, $S_\epsilon^{(i)}$, corresponding to each distinct $\lambda_i$ to be embedded can be standardised/rotated so that the quadrants (i.e. those regions between bad cones) where the potential, $\widetilde{q}_k$, needs to be positive in order for the respective eigenvector to shrink are the same quadrants as for the other eigenvectors to be embedded (see Figure~\ref{fig51} for more details). This results in rotating the vector with the initial components $(u_0,u_1)^T$ about the plane; however, since all the lemmas in this paper are independent of the location of any of the starting vectors then the results remain valid. Consequently, we will assume all $S_\epsilon^{(i)}$ are equal to some standard $S_\epsilon$ and then it is possible to visualise the various rotations by different angles $\theta(\lambda_i)$, acting in the same plane to avoid the same cone, $S_\epsilon$. This implies that for all candidate eigenvector solutions to simultaneously receive sufficient shrinkage, it is sufficient to rotate the relevant vector components into the same quadrant of the plane as those for the other eigenvalues, or into quadrants diametrically opposite. This leads to the following definition for the acceptable region, $A_\epsilon$, which will be used throughout the rest of this paper.

\begin{definition}\label{5.10}
Let $\vec{f}_i\in\mathbb{R}^2$ for $i\in\{1,\dots,n\}$. The collection of vectors $(\vec{f}_1,\dots,\vec{f}_n)\in A_\epsilon$ with $A_\epsilon\subseteq\mathbb{R}^{2n}$ iff for all $i,j\in\{1,\dots,n\}$ we have $f_i\not \in S_{\epsilon}$ and $Q\vec{f}_i\equiv Q\vec{f}_j \mod 2,$ where $Q:\mathbb{R}^2\mapsto \{1,2,3,4\}$, where $\{1,2,3,4\}$ correspond to the quadrants in Figure~\ref{5.72}. Informally, a collection of vectors belong to $A_\epsilon$ if they reside in compatible regions of the plane.
\end{definition}

 As an example consider the vectors $\vec{f}_1,\vec{f}_2,\vec{f}_3,\vec{f}_4$ in Figure~\ref{5.72}. The collection $\left(\vec{f}_1,\vec{f}_3\right)\in A_\epsilon$, as the vector $\vec{f}_1$ inhabits quadrant $1$ and $\vec{f}_3$ inhabits quadrant $3$, and $1\equiv 3 \mod 2$. For similar reasons $\left(\vec{f}_2,\vec{f}_4\right)\in A_\epsilon$. However the collection $\left(\vec{f}_1,\vec{f}_2\right)\not\in A_\epsilon$ as $\vec{f}_1$ inhabits quadrant $1$ and $\vec{f}_2$ inhabits quadrant $2$ and $1\not\equiv 2 \mod 2$. The collection $\left(\vec{f}_3,\vec{f}_4\right)$ does not belong to $A_\epsilon$ by a similar argument.

\begin{figure}
\centering
\begin{tikzpicture}
   \draw [ultra thick, red, -] (0,-2) -- (0,2)
        node [right,black] at (-3,2) {Before Stand.}
        node  [left, black] at (2,2) {$\lambda_1$}    % draw y-axis line
        node [left, black] at (-1/2,1) {\Huge +}
        node [right,black] at (1/2,1) {\Huge $-$}
         node [left, black] at (-1/2,-1) {\Huge $-$}
        node [right,black] at (1/2,-1) {\Huge +};              % add label for y-axis
    \draw [ultra thick, red, -] (-2,0) -- (2,0);

    \draw [ultra thick, red, -] (3,2) -- (7,-2)
         node  [left, black] at (7,2) {$\lambda_2$}
        node [right, black] at (6,0) {\Huge $-$}
        node [above,black] at (5,1) {\Huge $+$}
         node [below, black] at (5,-1) {\Huge $+$}
        node [left,black] at (4,0) {\Huge $-$};              % add label for y-axis
    \draw [ultra thick, red, -] (7,2) -- (3,-2);      % draw x-ax

    \draw [ultra thick, red, -] (0,-3) -- (0,-7)
        node [right,black] at (-3,-3) {After Stand.}
        node  [left, black] at (2,-3) {$\lambda_1$}    % draw y-axis line
        node [left, black] at (-1/2,-4) {\Huge +}
        node [right,black] at (1/2,-4) {\Huge $-$}
         node [left, black] at (-1/2,-6) {\Huge $-$}
        node [right,black] at (1/2,-6) {\Huge +};              % add label for y-axis
    \draw [ultra thick, red, -] (-2,-5) -- (2,-5);      % draw x-axis line

    \draw [ultra thick, red, -] (5,-3) -- (5,-7)
        node  [left, black] at (7,-3) {$\lambda_2$}    % draw y-axis line
        node [left, black] at (5-1/2,-4) {\Huge +}
        node [right,black] at (5+1/2,-4) {\Huge $-$}
         node [left, black] at (5-1/2,-6) {\Huge $-$}
        node [right,black] at (5+1/2,-6) {\Huge +};              % add label for y-axis
    \draw [ultra thick, red, -] (3,-5) -- (7,-5);      % draw x-axis line
\end{tikzpicture}
\caption{The thick axes illustrate the arbitrarily narrow cones to be avoided for two candidate eigenvalues $\lambda_1,\lambda_2$ respectively, with the first row representing the situation before any attempt at standardization has been made, and the second afterwards, specifically once the orthogonal bad cones corresponding to $\lambda_2$ have been appropriately rotated. The $+,-$ signs in each acceptable quadrant indicate the sign that the relevant non-zero component of the potential, $\widetilde{q}_k$, must take in order for the candidate eigenvector to shrink sufficiently. Observe that after the standardization has been made (i.e. the second row of axes) the same quadrant can be chosen for both $\lambda_1$ and $\lambda_2$ and it will be consistent in terms of the sign required of the potential component, $\widetilde{q}_k$, to produce shrinkage.}
\label{fig51}
\end{figure}
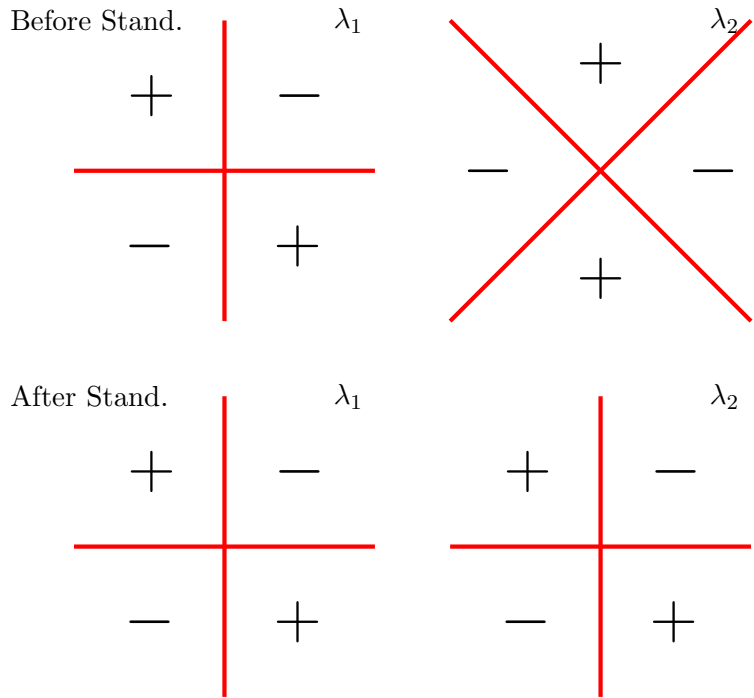

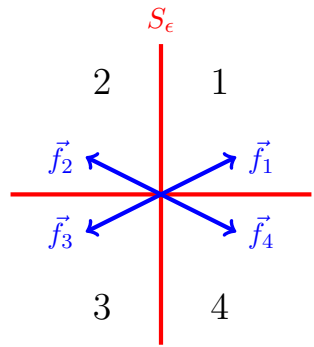
\begin{figure}
\centering
\begin{tikzpicture}
    \draw [ultra thick, red, -] (0,-2) -- (0,2)
        node  [above, red] at (0,2) {$S_\epsilon$}    % draw y-axis line
        node [left, black] at (-1/2,1.5) {\Large 2}
        node [right,black] at (1/2,1.5) {\Large 1}
         node [left, black] at (-1/2,-1.5) {\Large 3}
        node [right,black] at (1/2,-1.5) {\Large 4};              % add label for y-axis
    \draw [ultra thick, red, -] (-2,0) -- (2,0);
    \draw [ultra thick, blue, ->] (0,0) -- (-1,2/4)
   node  [right, blue] at (1,2/4) {$\vec{f}_1$}
   node  [left, blue] at (-1,2/4) {$\vec{f}_2$}
   node  [right, blue] at (1,-2/4) {$\vec{f}_4$}
   node  [left, blue] at (-1,-2/4) {$\vec{f}_3$}     ;
    \draw [ultra thick, blue, ->] (0,0) -- (1,2/4);
    \draw [ultra thick, blue, ->] (0,0) -- (1,-2/4);
    \draw [ultra thick, blue, ->] (0,0) -- (-1,-2/4);      % draw x-axis line

\end{tikzpicture}
\caption{The thick axes illustrate $S_\epsilon$, the arbitrarily narrow cones to be avoided. The numbers $1,2,3,4$ label the acceptable quadrants between the cones and are used in Definition~\ref{5.10}, whilst the vectors $\vec{f}_1,\vec{f}_2,\vec{f}_3,\vec{f}_4$ are employed in the discussion that follows the definition. }
\label{5.72}
\end{figure}

\begin{theorem}\label{5.45}
Let $J_T$ be an arbitrary period-$T$ Jacobi operator and $\{\lambda_i\}_{i=1}^\infty$ a sequence of complex numbers belonging to $\sigma_{ell}(J_T)$, where any finite collection $$\{\pi,\theta(\lambda_{1}),\theta(\lambda_{2}),\dots,\theta(\lambda_{k})\}$$ is rationally independent. Then for any positive sequence $K(n)$ with $K(n)\rightarrow\infty$ arbitrarily slowly as $n\rightarrow\infty$, there exists a potential $|q_n|\leq \frac{K(n)}{n}$ for all $n$, non-zero vectors $(u_{n,i})_{n\geq1}\in l^2(\mathbb{N};\mathbb{R})$ and an infinite diagonal matrix, $Q$, with entries $(q_n)$,  such that $$(J_T+Q)(u_{n,i})_{n\geq1}=\lambda_i (u_{n,i})_{n\geq1}.$$
\end{theorem}

 \begin{proof}
 The ability to embed infinitely many eigenvalues using a result (Proposition~\ref{5.14}) that is only valid for finitely many eigenvalues follows by breaking up the set $\{\lambda_i\}_{i=1}^\infty$ to be embedded into an increasing sequence of finite subsets, $N_k$, such that $$N_1\subset N_2\subset N_3\subset\dots \subset N_k\subset\dots$$ and $$\bigcup\limits_{k=1}^\infty N_{k}=\{\lambda_i\}_{i=1}^\infty.$$ Then at each particular stage we are only ever dealing with finitely many eigenvalues (with quasi-momenta rationally dependent with each other and $\pi$) which Proposition~\ref{5.14} is sufficient to deal with. For instance, consider a possible sequence $N_k$ where $$N_1=\{\lambda_1\},N_2=\{\lambda_1,\lambda_2\},N_3=\{\lambda_1,\lambda_2,\lambda_3\},\dots.$$ Then for the first stage of our calculation we are concerned only with embedding one eigenvalue, $\lambda_1$, so that we deal with rotating the initial components corresponding to this candidate eigenvector, $(u_{n,1})$  into a shrinkable area of the plane, and applying an appropriate potential $\widetilde{q}_1$ once there for the reduction to take effect. (As in the single eigenvalue case we denote the number of rotations necessary by $r_1$). Then move onto $N_2=\{\lambda_1,\lambda_2\}$ and use Proposition~\ref{5.14} to find a bound, $r_2$, on the number of rotations necessary to manoeuvre the appropriate components of both candidate eigenvectors, $(u_{n,1}),(u_{n,2})$, simultaneously into consistent shrinkable areas of the plane, $A_\epsilon$, and applying a potential $\widetilde{q}_2$. The process is then continued by considering the set $N_3$, et cetera.

 Consequently, the proof rests on showing that an eigenvector, $\underline{u}$, corresponding to an arbitrary $\lambda\in\{\lambda_i\}_{i=1}^\infty$, belongs to the sequence space $l^2(\mathbb{N};\mathbb{R})$; the idea being somewhat similar to the three-step single eigenvalue technique. However, now the eigenvalue corresponding to this eigenvector might not necessarily be in the set $N_1$ (the set of eigenvalues whose components are rotated into their respective cones straight away). Instead, we assume $\lambda\in N_i$ for all $i\geq k$ and $\lambda\not\in N_j$ for all $j<k$ where $k\in\mathbb{N}$. In particular $\lambda$ first appears in the set $N_k$, and we denote the contribution to the square of the norm from the initial components as $A_k(\lambda)$. Then continuing as before in the single eigenvalue case one uses Proposition~\ref{5.14} to ensure that all the vectors $\vec{f}_i$ corresponding to the elements of $N_j$ are simultaneously rotated into shrinkable regions of the plane.

Thus, the norm of the eigenvector solution can be estimated by
\begin{multline}\label{4.12}
\|\underline{u}(\lambda)\|^2\leq A_k(\lambda)+D_k(\lambda)\sum\limits_{m=k}^\infty\left[ \prod\limits_{j=k}^m (1-C(\lambda)|\widetilde{q}_j|)\right](R(N_{m+1})+1)\\+F_k(\lambda)\sum\limits_{t=k}^\infty\left[\prod\limits_{s=k}^{t-1}(1-C(\lambda)|\widetilde{q}_s|)\right],
\end{multline} where $A_k(\lambda),D_k(\lambda),F_k(\lambda)\in\mathbb{R},\widetilde{q}_0:=0$ and $R(N_{m})$ corresponds to the number of rotations necessary to simultaneously rotate the collection of vectors corresponding to elements of the set $N_m$ into consistent acceptable regions of the plane, $A_\epsilon$.

It remains to show that this sum is bounded; in particular that the vector, $\underline{u}$, belongs to the sequence space $l^2(\mathbb{N};\mathbb{R})$.
Now, we use an approach similar to that used in the proof of Theorem 1 in \cite{9g} and reproduce some of the intermediate steps. Letting $|\widetilde{q}_m|=\frac{L_m}{m}$, where $L_m$ will be specified below such that $L_m\rightarrow\infty$ as $m\rightarrow\infty$, and $\frac{L_m}{m}$ is square-summable, we observe that
\begin{align*}
&\left[\prod_{j=k}^m\left(1-C(\lambda)|\widetilde{q}_j|\right)\right]\leq \exp\left(\sum\limits_{j=k}^m\log(1-C(\lambda)|\widetilde{q}_j|)\right)\\
&\leq \exp\left(-C(\lambda)\sum\limits_{j=k}^m |\widetilde{q}_j|\left(1+O\left(\widetilde{q}_j\right)\right)\right)\leq H_k(\lambda) \exp\left(-C(\lambda)\sum\limits_{j=1}^m |\widetilde{q}_j|\right),
\end{align*} where $H_k(\lambda)$ is a function of $\lambda$ and absorbs the contribution from the $\widetilde{q}_k^2$ terms.

Unlike in the single eigenvalue case, the term $R(N_{k})$ is now no longer constant and in fact grows. However, by choosing the sets $N_k$ appropriately the term $R(N_{k})$ can be controlled so as to be of order $\sqrt{k}$. This can be done by repeating the sets $N_i$ sufficiently often, i.e. letting $N_m=N_{m+1}=N_{m+2}$ for sufficiently many steps. Then,
\begin{align*}
&\sum\limits_{m=1}^\infty(R(N_{m+1})+1)\left[\exp\left(-C(\lambda)\sum\limits_{j=1}^m |\widetilde{q}_j|\right)\right]\\
&\leq \sum\limits_{m=1}^\infty(1+\sqrt{m})\left[\exp\left(-C(\lambda)\sum\limits_{j=1}^m \frac{L_j}{j}\right)\right]\\
&\leq \widetilde{H}_s(\lambda) \sum\limits_{m=s}^\infty (1+\sqrt{m}) m^{-C(\lambda)L_s}\leq \widetilde{\widetilde{H}}_s(\lambda)\sum\limits_{m=1}^\infty (1+\sqrt{m}) m^{-2}
\end{align*} since $C(\lambda)L_s>2$ for $s$ sufficiently large and where $\widetilde{H}_s(\lambda)$ and $\widetilde{\widetilde{H}}_s(\lambda)$ are functions of $\lambda$. This sum is convergent and therefore \eqref{4.12} is convergent, meaning $\lambda$ is an embedded eigenvalue.

Now let the sequence $K(n)$ be as stated in the theorem, with $K(n)\rightarrow\infty$ as $n\rightarrow\infty$. It still remains to check that $|q_n|\leq \frac{K(n)}{n}$, for all $n$. Without loss of generality we may replace the sequence $K(n)$ by another sequence (again denoted by $K(n)$ to avoid complicated notation) which has the property that $K(n)\rightarrow\infty$ as slowly as we need for the proof and that $K(n)$ is monotonically increasing. Note that the explicit condition on the `new' $K(n)$ appears in the step-by-step construction of the potential. The only thing we need is that the new $K(n)$ is subordinated to the $K(n)$ given in the theorem. Recall $|\widetilde{q}_n|=\frac{L_n}{n}$ and let $L_n:=\left(K(\widetilde{n})\right)^{\frac{1}{3}}$, where $\widetilde{n}$ is the position of the $n$-th non-zero entry in the potential. Assuming the sequence $K(n) \rightarrow\infty$ sufficiently slowly we can guarantee that the sequence $\frac{L_n}{n}$ is square-summable. Moreover, $$\widetilde{n} = 1-T + \sum\limits_{k=1}^n T(r_k + 1)\leq 2T\sum\limits_{k=1}^n R\left(N_k\right),$$ where we recall that $r_i$ denotes the actual number of applications of the monodromy matrix producing `rotations' (as opposed to the upper bound) required to move the initial components into an acceptable region of the plane for the $i$-th shrinkage after the $(i-1)$-th shrinkage.

By the monotonicity of $K(n)$,
\begin{equation}\label{4.14}|\widetilde{q}_n| = \frac{L_n}{n}=\frac{(K(\widetilde{n}))^{\frac{1}{3}}}{n}\leq \frac{\left(K\left(2T\sum\limits_{k=1}^nR(N_k)\right)\right)^{\frac{1}{3}}}{n}.\end{equation} Furthermore, we can, without loss of generality, assume $\frac{K(n)}{n}$ is monotonically decreasing. It is easy to see that for $K(n)$ increasing sufficiently slowly we may combine this condition with monotonically increasing $K(n)$. Therefore,\begin{equation}\label{4.13}\frac{K(\widetilde{n})}{\widetilde{n}}\geq \frac{K\left(2T\sum\limits_{k=1}^nR(N_k)\right)}{2T\sum\limits_{k=1}^nR(N_k)}.\end{equation} Using \eqref{4.14} and \eqref{4.13}, we see that the only non-zero values of the potential $|q_{\widetilde{n}}|=|\widetilde{q}_n|\leq \frac{K(\widetilde{n})}{\widetilde{n}}$ for $n\gg 1$ if \begin{equation}\label{4.15}
 \frac{\left(K\left(2T\sum\limits_{k=1}^nR(N_k)\right)\right)^{\frac{1}{3}}}{n} \leq \frac{K\left(2T\sum\limits_{k=1}^nR(N_k)\right)}{2T\sum\limits_{k=1}^nR(N_k)}.\end{equation}It should be stressed that the condition $|q_{\widetilde{n}}|\leq \frac{K(\widetilde{n})}{\widetilde{n}}$ would be enough to guarantee the whole potential estimate claimed in the theorem.

To establish when \eqref{4.15} is satisfied we see that if $N_k$ grows slowly enough then we have that $2TR(N_m) < \left(K(m)\right)^{\frac{2}{3}}$ for $m \gg 1$ so that $$2T\sum\limits_{j=1}^n R(N_j)<\sum\limits_{j=1}^n \left(K(j)\right)^{\frac{2}{3}}\leq n \left(K(n)\right)^{\frac{2}{3}}\leq n \left(K\left(2T \sum\limits_{j=1}^n R(N_j) \right)\right)^\frac{2}{3} $$ for $n\gg1$. By substituting this into the denominator on the right-hand side of \eqref{4.15} we see that the inequality is satisfied, for $n\gg 1$, and thus there is some constant $M$ such that for all $n\geq M$ the result $|q_n|\leq \frac{K(n)}{n}$ applies. We set $q_n\equiv 0$ for all $n<M$ and then construct the potential as explained in the text to obtain the result for all $n$. Note that our step-by-step construction of the potential will not be seriously affected by this correction of the potential on a finite interval. Although the potential slightly changes even after this finite interval (due to the change of the `rotation' intervals) one can easily see that all the estimates from above are preserved.

\end{proof}

\section{Example: Infinitely many eigenvalues with one or two exceptional values.}\label{sec3}

First we explore how to embed infinitely many eigenvalues (any finite selection of which have quasi-momenta that are rationally independent with each other and $\pi$) simultaneously with a single eigenvalue, $\lambda_1$, whose quasi-momentum, $\theta(\lambda_1)$, is specified to be rationally dependent with $\pi$, but not equal to $\frac{\pi}{2}$.

The following lemma will be needed.

\begin{lemma}\label{5.52}
Let $\theta(\lambda_1)=\frac{p\pi}{q}$ where $p,q\in\mathbb{N}$ and $\gcd(p,q)=1$, and let $$\{\pi, \theta(\lambda_2),\theta(\lambda_3),\dots,\theta(\lambda_n)\}$$ be rationally independent. Then for any collection of non-zero real vectors $\{\vec{f}_1,\dots,\vec{f}_n\}$ there exists $t\in\mathbb{N}$ such that the collection $$\left(R(t\theta(\lambda_1))\vec{f}_1,\dots,R(t\theta(\lambda_n))\vec{f}_n\right)\in A_\epsilon,$$ where $A_\epsilon$ is as defined in Definition~\ref{5.10}. In particular, the vectors, $\vec{f}_1,\dots, \vec{f}_n$ can be simultaneously rotated into the same quadrant or diametrically opposite ones.
\end{lemma}

\begin{proof}
First consider $\theta(\lambda_1)$. By Lemma~\ref{5.5} there exists a number, $k$, such that the new vector $R(k\theta(\lambda_1))\vec{f}_1$ is in $\mathbb{R}^2\setminus S_\epsilon$. Now, since the angle, $\theta(\lambda_1)$, associated with this vector is of the form $\frac{p\pi}{q}$, every subsequent $2q$ rotations will return us to the same point in the plane. Thus, create new angles, $\widetilde{\theta}(\lambda_i):=2q\theta(\lambda_i)$ for all $i\in\{2,\dots,n\}$. These new angles are still rationally independent, and consequently Proposition~\ref{5.14} can be applied to give an upper bound, $r$, on the number of rotations necessary to move all the $R(k\widetilde{\theta}(\lambda_2))\vec{f}_2, \dots,R(k\widetilde{\theta}(\lambda_n))\vec{f}_n$ into the same quadrant where $R(k\theta(\lambda_1))\vec{f}_1$ resides. Note that the purpose of these new angles is to ensure that every new step is $2q$ old steps, implying that the vector $R(k\theta(\lambda_1))\vec{f}_1$ remains fixed under subsequent rotations. The upper bound for the number of rotations is $k+2qr$. Thus, for some $t\leq k+2qr,$ we have $\left(R(t\theta(\lambda_1))\vec{f}_1,\dots,R(t\theta(\lambda_n))\vec{f}_n\right)\in A_\epsilon$.~\qedhere
\end{proof}

\begin{theorem}\label{5.60}
Let $J_T$ be an arbitrary period-$T$ Jacobi operator and $\{\lambda_i\}_{i=2}^\infty$ a sequence of numbers belonging to $\sigma_{ell}(J_T)$ where any collection $$\{\pi,\theta(\lambda_{2}),\theta(\lambda_{3}),\dots,\theta(\lambda_{n})\},~~~n\geq 2,$$ is rationally independent and $\lambda_1\in\sigma_{ell}(J_T)$ such that $\theta(\lambda)=\frac{p\pi}{q}\neq\frac{\pi}{2}$ with $p,q\in\mathbb{N},\gcd(p,q)=1$. Then for any positive sequence $K(n)$ with $K(n)\rightarrow\infty$ arbitrarily slowly as $n\rightarrow\infty$, there exists a potential $|q_n|\leq \frac{K(n)}{n}$ for all $n$, and non-zero vectors $(u_{n,i})_{n\geq1}\in l^2(\mathbb{N};\mathbb{R})$ such that $$(J_T+Q)(u_{n,i})_{n\geq1}=\lambda_i (u_{n,i})_{n\geq 1}$$ for all $i$ and where $Q$ is an infinite diagonal matrix with entries $(q_n)$.
\end{theorem}

 \begin{proof}
We wish to show that an arbitrary eigenvector, $\underline{u}$, belongs to the sequence space $l^2(\mathbb{N};\mathbb{R})$. Moreover, it is required that $\lambda_1$ appear in the set $N_1$, where $\bigcup_{i=1}^\infty N_i=\{\lambda_i\}_{i=1}^\infty$ and $N_{i}\subseteq N_{i+1}$ as before. This means that at every stage we are dealing with embedding one eigenvalue with a quasi-momentum rationally dependent with $\pi$, along with finitely many others eigenvalues whose quasi-momenta are rationally independent with each other and $\pi$. Lemma~\ref{5.52} is now invoked instead of the Proposition~\ref{5.14} to ensure that all relevant vectors associated to the elements of the set $N_i$ are simultaneously rotated into shrinkable areas of the domain. Thus, we construct a candidate eigenvector solution whose norm can be estimated by
\begin{multline}
\|\underline{u}(\lambda)\|^2\leq A_k(\lambda)+D_k\sum\limits_{m=k}^\infty\left[ \prod\limits_{j=k}^m (1-C(\lambda)|\widetilde{q}_j|)\right](R(N_{m+1})+1)\nonumber\\
+F_k(\lambda)\sum\limits_{t=1}^\infty\left[\prod\limits_{s=k}^{t-1}(1-C(\lambda)|\widetilde{q}_s|)\right].
\end{multline}
We bound the growth of $R(N_{k})=O\left(\sqrt{k}\right)$ by increasing the sets $N_k$ sufficiently slowly, and then take logarithms and exponentials as in Section~\ref{sec1} to show that the candidate eigenvector belongs to the sequence space $l^2(\mathbb{N};\mathbb{R})$.~\qedhere
\end{proof}

We now discuss the case of embedding two eigenvalues, $\lambda_1,\lambda_2$, first where $0<\theta_1<\theta_2<\frac{\pi}{2}$ or $\frac{\pi}{2}<\theta_1<\theta_2<\pi$, and then where $0<\theta(\lambda_1)<\frac{\pi}{2}<\theta_2(\lambda_2)<\pi$.

\begin{lemma}\label{5.43}
Let $0<\theta_1<\theta_2<\frac{\pi}{2}$ or $\frac{\pi}{2}<\theta_1<\theta_2<\pi$. Then, for any pair of non-zero real vectors $\{\vec{f}_1,\vec{f}_2\}$ in the plane, there exists $k$ such that the collection $$\left(R(k\theta_1)\vec{f}_1,R(k\theta_2)\vec{f}_2\right)\in A_\epsilon.$$ In particular, the vectors, $\vec{f}_1,\vec{f}_2$ can be simultaneously rotated into the same quadrant or diametrically opposite ones.
\end{lemma}

\begin{proof}
First consider the case when $0<\theta_1<\theta_2<\frac{\pi}{2}$. Since $\theta_2>\theta_1$ then eventually there will be an `overtaking'. This means that the vector $\vec{f}_2$ will overtake the vector $\vec{f}_1$ at some stage. In the step preceding this overtaking the two vectors $\vec{f}_1,\vec{f}_2$ are either in the same quadrant or not. If they're in the same quadrant, then the objective is already achieved. If they're not then the vector $\vec{f}_2$ must be in the quadrant `behind' the quadrant $\vec{f}_1$ is in or in the cone; however after one more rotation the vector $\vec{f}_2$ overtakes $\vec{f}_1$, and since $\theta_2<\frac{\pi}{2}$ the two vectors must then be in the same quadrant (see Figure~\ref{fig54} for more details) and outside $S_\epsilon$ for sufficiently small $\epsilon$.

When $\frac{\pi}{2}<\theta_1<\theta_2<\pi$ the argument is similar. Again, the only bad situation is when $\vec{f}_2$ is immediately `behind' $\vec{f}_1$. Then, in the next step when the `overtaking' happens the vector $\vec{f}_1$ must move into the next quadrant, where it is joined by the vector $\vec{f}_2$. It is not possible for the vector $\vec{f}_2$ to `overshoot' and land in the inconsistent quadrant beyond $\vec{f}_1$ as this would require $\vec{f}_2$ to move more than two whole quadrants in one rotation, and we already assume $\theta_2<\pi$ (see Figure~\ref{fig56} for more details).~\qedhere
\end{proof}

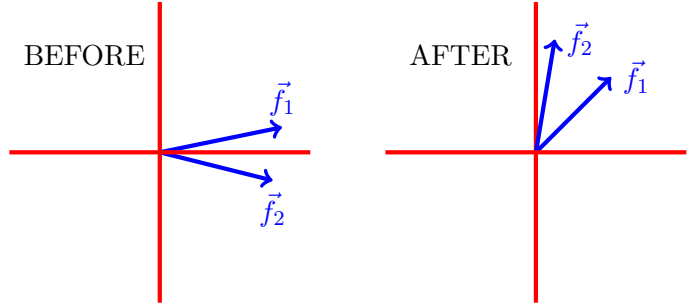
\begin{figure}
\centering
\begin{tikzpicture}
        \draw [ultra thick, red, -] (0,-2) -- (0,2)
        node[above, black] at (-1,1) {BEFORE};
         \draw[ultra thick, blue, ->] (0,0) -- (1+2/4,-3/8)
         node[below, blue] at (1+2/4,-3/8) {$\vec{f}_2$};
         \draw[ultra thick, blue, ->] (0,0)-- (1+5/8,1/3)
         node[above, blue] at (1+5/8,1/3) {$\vec{f}_1$};
    \draw [ultra thick, red, -] (-2,0) -- (2,0);

    \draw [ultra thick, red, -] (5,-2) -- (5,2)
    node[above, black] at (4,1) {AFTER};
    \draw [ultra thick, blue, ->] (5,0) -- (5+1/4,1+1/2)
    node[right, blue] at (5+1/4,1+1/2) {$\vec{f}_2$};
     \draw [ultra thick, blue, ->] (5,0) -- (6,1)
      node[right, blue] at (6,1) {$\vec{f}_1$};
   \draw [ultra thick, red, -] (3,0) -- (7,0);      % draw x-axis line

\end{tikzpicture}
\caption{The diagram illustrates one particular example of `overtaking', specifically when $0<\theta_1<\theta_2<\frac{\pi}{2}$ and $\vec{f}_1,\vec{f}_2$  do not inhabit the same quadrant before the overtaking. The two images denote the location of the vectors before and after the threshold rotation has been applied. As can be seen, since $\theta_2<\frac{\pi}{2}$ there is no threat of $\vec{f}_2$ overshooting, and landing in the next (inconsistent) quadrant instead. Note that the thick  axes denote the arbitrarily narrow bad cones to be avoided. }
\label{fig54}
\end{figure}

\begin{figure}
\centering
\begin{tikzpicture}
        \draw [ultra thick, red, -] (0,-2) -- (0,2)
        node[above, black] at (-1,1) {BEFORE};
         \draw[ultra thick, blue, ->] (0,0) -- (1+2/4,-3/8)
         node[below, blue] at (1+2/4,-3/8) {$\vec{f}_2$};
         \draw[ultra thick, blue, ->] (0,0)-- (1+5/8,1/3)
         node[above, blue] at (1+5/8,1/3) {$\vec{f}_1$};
    \draw [ultra thick, red, -] (-2,0) -- (2,0);

    \draw [ultra thick, red, -] (5,-2) -- (5,2)
    node[above, black] at (6,1) {AFTER};
    \draw [ultra thick, blue, ->] (5,0) -- (3+7/8,3/4)
    node[left, blue] at (3+3/4,1/2){$\vec{f}_2$};
     \draw [ultra thick, blue, ->] (5,0) -- (4+1/3,1)
      node[above, blue] at (4+1/3,1) {$\vec{f}_1$};
   \draw [ultra thick, red, -] (3,0) -- (7,0);      % draw x-axis line

\end{tikzpicture}
\caption{The diagram illustrates another example of `overtaking', specifically when $\frac{\pi}{2}<\theta_1<\theta_2<\pi$ and $\vec{f}_1,\vec{f}_2$ do not inhabit the same quadrant before the overtaking. The two images denote the location of the vectors before and after the threshold rotation has been applied. As can be seen, since $\frac{\pi}{2}<\theta_2$ the vector $\vec{f}_2$ lands in the next diametrically opposite quadrant. Note that the thick axes denote the arbitrarily narrow bad cones to be avoided. }
\label{fig56}
\end{figure}
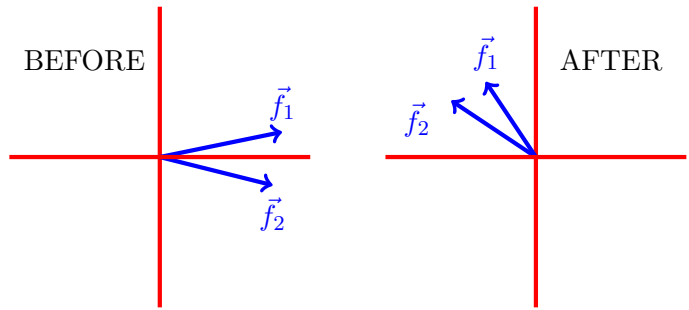

\begin{theorem}\label{5.57}
Let $J_T$ be an arbitrary period-$T$ Jacobi operator and $\{\lambda_i\}_{i=3}^\infty$ a sequence of complex numbers belonging to $\sigma_{ell}(J_T)$ where any collection $$\{\pi,\theta(\lambda_3),\theta(\lambda_{4}),\dots,\theta(\lambda_{n})\},~~n\geq 3$$ is rationally independent, and $\{\lambda_1,\lambda_2\}\subseteq\sigma_{ell}(J_T)$ where either $0<\theta(\lambda_1)<\theta(\lambda_2)<\frac{\pi}{2}$ or $\frac{\pi}{2}<\theta(\lambda_1)<\theta(\lambda_2)<\pi$. Then for any positive sequence $K(n)$ with $K(n)\rightarrow\infty$ arbitrarily slowly as $n\rightarrow\infty$, there exists a potential $|q_n|\leq \frac{K(n)}{n}$ for all $n$, and non-zero vectors $(u_{n,i})_{n\geq1}\in l^2(\mathbb{N};\mathbb{R})$ such that $$(J_T+Q)(u_{n,i})_{n\geq1}=\lambda_i(u_{n,i})_{n\geq1}$$ for all $i$ and where $Q$ is an infinite diagonal matrix with entries $(q_n)$.
\end{theorem}

\begin{proof}
For the case when only one of $\theta(\lambda_1),\theta(\lambda_2)$ is rationally dependent with $\pi$, we already have the result (see Theorem \ref{5.60}). Thus the only outstanding case is when $\theta(\lambda_1)=\frac{p_1\pi}{q_1}$ and $\theta(\lambda_2)=\frac{p_2\pi}{q_2}$ with $\gcd(p_1,q_1)=1=\gcd(p_2,q_2)$.

The argument follows similarly to the proof of Theorem~\ref{5.60} by dividing up the set of eigenvalues to be embedded into sets $N_1\subset N_2\subset N_3\subset\dots$ except now $\{\lambda_1,\lambda_2\}\subseteq N_1$. The `overtaking' argument used in Lemma~\ref{5.43} can be applied again to establish that a finite number of rotations will be enough to simultaneously manoeuvre the initial vectors in the plane corresponding to the candidate eigenvalues, $\lambda_1,\lambda_2$, into shrinkable regions. Moreover, for subsequent sets, $N_k$, it will become necessary to simultaneously rotate finitely many other initial vector components using quasi-momenta $\theta(\lambda_{n_1}),\dots,\theta(\lambda_{n_k})$, rationally independent with each other and $\pi$. Specifically, we first set about moving the vectors, $\vec{f}_1, \vec{f}_2$ corresponding to the two quasi-momenta rationally dependent with $\pi$ into shrinkable regions of the plane, just as in the $N_1$ case, with the upper bound $n_0$ on the number of rotations required. Then, observing that for every subsequent $t:={\rm LCM}(q_1,q_2)$ rotations the vectors $\vec{f}_1,\vec{f}_2$ remain in the shrinkable region of the plane, create new  quasi-momenta $\widetilde{\theta}(\lambda_i)=t\theta(\lambda_i)$ for $i\in \mathbb{N}\setminus\{1,2\}$, rationally independent with each other and $\pi$. Then, Proposition~\ref{5.14} can again be applied to the vectors $\vec{f}_i$ for $i\in I_k$, where $I_k:=\{i\in\mathbb{N}\setminus\{1,2\}:\lambda_i\in N_k\}$, to rotate them all simultaneously into the same shrinkable region as $\vec{f}_1,\vec{f}_2$ but now using the new angles $\widetilde{\theta}_i$ since this ensures that at every time a rotation is applied the vectors $\vec{f}_1,\vec{f}_2$ remain in the same shrinkable region. This technique is applicable for all sets $N_k$.

Thus, for each $\lambda_i$, we construct a candidate eigenvector solution that satisfies
\begin{multline}
\|\underline{u}(\lambda)\|^2\leq A_k(\lambda)+D_k(\lambda)\sum\limits_{m=k}^\infty\left[ \prod\limits_{j=k}^m (1-C(\lambda)|\widetilde{q}_j|)\right](R(N_{m+1})+1)\\+F_k(\lambda)\sum\limits_{t=1}^\infty\left[\prod\limits_{s=k}^{t-1}(1-C(\lambda)|\widetilde{q}_s|)\right],
\end{multline}
Again we slow the growth of $R(N_{k})=O\left(\sqrt{k}\right)$ and then take logarithms and exponentials as before to show that the candidate eigenvector belongs to the sequence space $l^2(\mathbb{N};\mathbb{R})$.~\qedhere
\end{proof}

We now continue to look  at embedding two eigenvalues, simultaneously, but this time consider the more sophisticated case when the quasi-momenta are such that  $0<\theta(\lambda_1)<\frac{\pi}{2}<\theta(\lambda_2)<\pi$.

\begin{lemma}\label{5.71}
Let $0<\theta_1<\frac{\pi}{2}<\theta_2<\pi$, but where $\pi-\theta_2\neq\theta_1$. Then for any pair of non-zero real vectors $\{\vec{f}_1,\vec{f}_2\}$ in the plane there exists $k$ such that the collection $$\left(R(k\theta_1)\vec{f}_1,R(k\theta_2)\vec{f}_2\right)\in A_{\epsilon}.$$ In particular, the vectors $\vec{f}_1,\vec{f}_2$ can be simultaneously rotated into the same quadrant or diametrically opposite ones.
\end{lemma}

\begin{proof}
The technique uses the `overtaking' argument, like Lemma~\ref{5.43}. However, due to the comparative size of the angles involved there is now the threat that before and after the overtaking happens the vectors $\vec{f}_1$ and $\vec{f}_2$ are still in inconsistent quadrants. This is the only case  that needs to be considered. Define the new angle $\widetilde{\theta}_2:=\theta_2-\pi$ and without loss of generality consider the situation when $|\widetilde{\theta}_2|<\theta_1$. (The case $|\widetilde{\theta}_2|>\theta_1$ follows a similar argument, whilst the case $\left|\widetilde{\theta}_2\right|=\theta_1$ has been excluded by the acceptable conditions of the lemma.) Since we may identify opposite quadrants we can assume $\vec{f}_2$ is rotated by $\widetilde{\theta}_2$ rather than $\theta_2$.

For the sake of simplicity we skip straight to the step before the overtaking is to happen and assume the two vectors inhabit inconsistent quadrants (otherwise there is no problem). Moreover, after the overtaking we assume the vectors reside in inconsistent quadrants (otherwise again there is no problem). By identifying consistent quadrants, without loss of generality the problematic case occurs when $\xi_1\leq 0, {\xi}_2\geq 0, {\xi}_1+\theta_1\geq 0,{\xi}_2+\theta_2\leq 0$, where $\xi_i:=\arg(f_i)$. We aim to show that since $\theta_1>|\widetilde{\theta}_2|$ there exists a $k_1\in\mathbb{N}$ such that \begin{equation}\label{5.4} {\xi}_1+k_1\theta_1>\frac{\pi}{2}~~{\rm and}~~{\xi}_2+k_1\widetilde{\theta}_2>-\frac{\pi}{2},\end{equation} i.e. after $k_1$ steps $\vec{f}_1$ has been rotated into the second quadrant, while $\vec{f}_2$ is still in the consistent fourth quadrant (see Figure~\ref{fig58} for more details). These are satisfied if $$(k_1-1){\theta}_1>\frac{\pi}{2}~{\rm and}~~k_1<-\frac{\pi}{2\widetilde{\theta_2}},$$ respectively. For this to occur, it is sufficient for $\theta_1,\widetilde{\theta_2}$ be such that
$$\frac{\pi}{2\theta_1}+1<k_1<-\frac{\pi}{2\widetilde{\theta}_2}-1.$$ This implies that
\begin{align*}
&\pi\widetilde{\theta}_2+4\theta_1\widetilde{\theta}_2>-\pi\theta_1\\
&\iff \pi\left(\theta_1+\widetilde{\theta}_2\right)>-4\theta_1\widetilde{\theta}_2
\end{align*}
Note that we subtracted $1$ from the upper-bound since not only must $k$ exist it must also be a natural number. Indeed, if $\theta_1,\widetilde{\theta}_2$ are such that the relations hold for the new (reduced) upper bound then they will also hold for the original bound and since the difference between the two upper bounds is $1$ there must be some natural number $k$ in the range.

 If these conditions aren't met then there is still the chance that the same `consistent-coinciding', happens later, just in another quadrant. This happens when both are in the lower left quadrant, i.e. if there exists $k_2\in\mathbb{N}$ such that \begin{equation}{\xi}_1+k_2\theta_1>\pi~~{\rm{and}}~~{\xi}_2+k_2\widetilde{\theta}_2>-\pi.\end{equation} These are satisfied if
 $$k_2-1>\frac{\pi}{\theta_1}~{\rm and }~~ k_2<\frac{-\pi}{\widetilde{\theta}_2},$$ respectively. Then for this to occur it is sufficient for $\theta_1,\widetilde{\theta_2}$ to be such that
 $$1+\frac{\pi}{\theta_1}<k_2<\frac{\pi}{-\widetilde{\theta}_2}-1.$$ This implies that
\begin{align*}
&2\theta_1\widetilde{\theta}_2+\pi(\theta_1+\widetilde{\theta}_2)>0\\
&\iff \pi(\theta_1+\widetilde{\theta}_2)>-2\theta_1\widetilde{\theta}_2.
\end{align*}
Note again that we have subtracted $1$ from the upper-bound.
If such a $k_2$ does not exist then we can keep repeating the process. For $k_n$ we need $\theta_1,\widetilde{\theta_2}$ such that
$$1+\frac{n(\frac{\pi}{2})}{\theta_1}<\frac{n(\frac{\pi}{2})}{-\widetilde{\theta}_2}-1 \iff \frac{n\pi}{2}(\theta_1+\widetilde{\theta}_2)>-2\theta_1\widetilde{\theta}_2.$$
Then for any $|\widetilde{\theta}_2|<\theta_1$ this condition is satisfied for any $n$ large enough. For the case $|\widetilde{\theta}_2|>\theta_1$ the same expression is obtained.~\qedhere
\end{proof}

\begin{figure}
\centering
\begin{tikzpicture}
        \draw [ultra thick, red, -] (0,-3) -- (0,3);
                \draw[ultra thick, blue, ->] (0,0) -- (3,-1/2)
         node[right, blue] at (3,-1/2) {$\xi_1$};
         \draw[ultra thick, blue, ->] (0,0) -- (3,-2.5/2)
          node[right, blue] at (3,-2.5/2) {$\xi_2+\widetilde{\theta}_2$};
         \draw[ultra thick, blue, ->] (0,0) -- (1.8,-2)
         node[right, blue] at (1.8,-2) {$\xi_2+2\widetilde{\theta}_2$};
         \draw[ultra thick, blue, ->] (0,0) -- (3/8,-2.5)
         node[right, blue] at (3/8,-2.5) {$\xi_2+k_1\widetilde{\theta}_2$};
    \draw [ultra thick, red, -] (-3,0) -- (3,0);
      \draw[ultra thick, blue, ->] (0,0) -- (3,1/8)
      node[right, blue] at (3,1/8) {$\xi_2$};
   \draw[ultra thick, blue, ->] (0,0) -- (2+1/2,3/2)
   node[right, blue] at (2+1/2,3/2) {$\xi_1+\theta_1$};
   \draw[ultra thick, blue, ->] (0,0) -- (1,3)
   node[right, blue] at (1,3) {$\xi_1+2\theta_1$};
   \draw[ultra thick, blue, ->] (0,0) -- (-1,3)
   node[left, blue] at (-1,3) {$\xi_1+k_1\theta_1$};
   \node[below, blue] at (0,3/2) {$\dots$};
   \node[right, blue] at (2/8,-3/2) {$\dots$};
\end{tikzpicture}
\caption{The diagram illustrates the possibility of the vector $\vec{f}_1$ being rotated into the the top-left quadrant whilst the vector $\vec{f}_2$ is still tarrying in the bottom-right. The expression $\xi_i+z\theta_i$ corresponds to the angle of the vector being described, whilst the thick axes are as in Figure~\ref{fig54}. }
\label{fig58}
\end{figure}
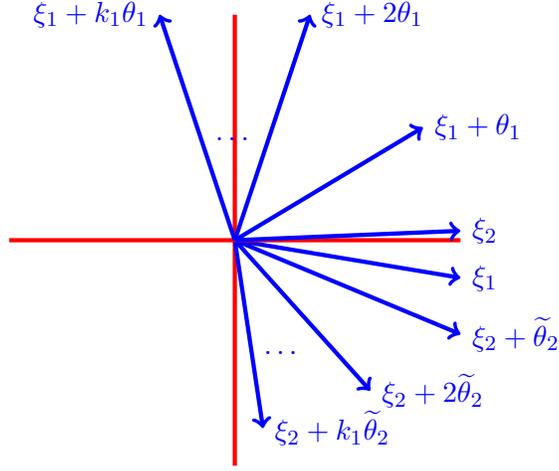

\begin{theorem}\label{5.73}
Let $J_T$ be an arbitrary period-$T$ Jacobi operator and $\{\lambda_i\}_{i=3}^\infty$ a sequence of complex numbers belonging to $\sigma_{ell}(J_T)$ where any collection $$\{\pi,\theta(\lambda_{3}),\theta(\lambda_{4}),\dots,\theta(\lambda_{n})\},~~n\geq 3$$ is rationally independent, and $\{\lambda_1,\lambda_2\}\subseteq\sigma_{ell}(J_T)$ where $0<\theta(\lambda_1)<\frac{\pi}{2}<\theta(\lambda_2)<\pi$ and $\theta(\lambda_1)\neq \pi-\theta(\lambda_2)$. Then for any positive sequence $K(n)$ with $K(n)\rightarrow\infty$ arbitrarily slowly as $n\rightarrow\infty$, there exists a potential $|q_n|\leq \frac{K(n)}{n}$ for all $n$, and non-zero vectors $(u_{n,i})_{n\geq1}\in l^2(\mathbb{N};\mathbb{R})$ such that $$(J_T+Q)(u_{n,i})_{n\geq1}=\lambda_i (u_{n,i})_{n\geq1}$$ for all $i$ and where $Q$ is an infinite diagonal matrix with entries $(q_n)$.
\end{theorem}

\begin{proof}
The argument is the same as in Theorem~\ref{5.57}, except now Lemma~\ref{5.43} is replaced by Lemma~\ref{5.71}.~\qedhere\end{proof}

\section{Infinitely many eigenvalues with a finite set of exceptional values.}\label{sec5}

Here, alongside infinitely many eigenvalues whose quasi-momenta are rationally independent with each other and $\pi$, we consider an arbitrary (but finite) selection of eigenvalues, where $\lambda_1, \lambda_2, \dots, \lambda_n$ such that $$\theta(\lambda_1)=\frac{p_1\pi}{q_1},\theta(\lambda_2)=\frac{p_2\pi}{q_2},\dots,\theta(\lambda_n)=\frac{p_n\pi}{q_n},$$ $p_i,q_i\in\mathbb{N},\gcd(p_i,q_i)=1,\gcd(q_i,q_j)=1$ for all $i\neq j$ and $\theta(\lambda_i)\neq\frac{\pi}{2}$ for all $i$.

The following elementary result will be needed:

\begin{proposition}\label{5.61} Let $\vec{f}$ be an arbitrary vector in the plane and $\theta=\frac{p\pi}{q}$ where $p,q\in\mathbb{N},\gcd(p,q)=1$. Then, if $2|p$, i.e. $2$ divides $p$, the orbit of $\vec{f}$ has $q$ distinct elements under the rotation matrix $R(\theta)$, all separated evenly by an angle of $\frac{2\pi}{q}$. If $2\centernot | p$ then the orbit of $\vec{f}$ has $2q$ distinct elements under the rotation matrix $R(\theta)$, all uniformly distributed with separating angle $\frac{\pi}{q}$.
\end{proposition}

 We can now  prove the following lemma which will be used in the main theorem of this section.

\begin{lemma}\label{5.62}
Let $\vec{f}$ be an arbitrary non-zero vector in the plane. Then for $\theta=\frac{p\pi}{q}\not\in\{\frac{\pi}{2},\frac{2\pi}{3}\}$ where $p,q\in\mathbb{N},\gcd(p,q)=1$, every quadrant will feature at least once in the orbit of $\vec{f}$ under the action of $R(\theta)$. \end{lemma}

\begin{proof}
First consider the case when $2\centernot | p$. Then, by Proposition~\ref{5.61}, the orbit of $\vec{f}$ has $2q$ distinct elements, evenly distributed about the circle of radius $|\vec{f}|$. Thus, if $q\geq 4$ then even if all the axes are hit, there will still be at least four other points in the orbit evenly distributed within the quadrants. If $q=3$ then there are two options: either no axis is hit, in which case there are six points of the orbit evenly distributed about the circle of radius $|\vec{f}|$, four of which must visit every quadrant; or at least one axis is hit by the orbit. However, if one axis is hit then the opposite axis is hit $3$ rotations later (or earlier), thus meaning at least two two axes are hit. However, no other axis can be hit because this would be at a distance of $\frac{\pi}{2}$ from either axis, and this is not a multiple of $\frac{\pi}{3}$. Thus, there are four points left in the orbit, meaning every one of the four quadrants is visited.

Secondly, consider the case when  $2|p$. Then, by Proposition~\ref{5.61}, the orbit of $\vec{f}$ has $q$ distinct elements, evenly distributed about the circle of radius $|\vec{f}|$. Thus, for $q\geq 8$ even if all four axes are hit, there are still four orbits left with which to visit every quadrant. Now since $q\in\{2,3,4,6\}$ are not valid options, we direct our attentions to $q=5$. The problem becomes tantamount to looking at a regular pentagon and seeing that at most only one vertex can lie on the axes, thus leaving four to visit every quadrant. Finally, for $q=7$, it is sufficient to look at a regular heptagon and observe that at most one vertex can reside on the axes.~\qedhere
\end{proof}

\begin{remark}
The invalidness of the result for the angle $\theta=\frac{\pi}{2}$ follows from the fact that if $\vec{f}_i$ begins on the central-axis of a bad cone then no quadrant is visited over its orbit. Similarly, for $\theta=\frac{2\pi}{3}$ the result fails because at most three quadrants are visited (since its orbit only has three entries) and at worst two (when the one of the orbit includes the axis).
\end{remark}

\begin{lemma}\label{5.63}
Let $\vec{f}_1,\dots,\vec{f}_n$ be any non-zero collection of vectors in the plane and $\theta_1=\frac{p_1\pi}{q_1},\theta_2=\frac{p_2\pi}{q_2},\dots,\theta_n=\frac{p_n\pi}{q_n}$ where $p_i,q_i\in\mathbb{N},\gcd(p_i,q_i)=1$, $\gcd(q_i,q_j)=1$ for all $i\neq j$ and $\theta_i\neq\frac{\pi}{2}$ for any $i$. Then there exists $k\in\mathbb{N}$ such that the collection $$\left(R(k\theta_1)\vec{f}_1,\dots,R(k\theta_n)\vec{f}_n\right)\in A_\epsilon.$$
\end{lemma}

\begin{proof}
(Case One) For all $\theta_i\neq \frac{2\pi}{3}$, we know that from Lemma~\ref{5.62} there exists some $a_i$ such that $R(a_i\theta_i)\vec{f}_i\not\in S_\epsilon$. What we are now aiming to do is rotate all vectors $\vec{f}_i$ into the same quadrant, i.e. we wish to find some $x$ such that
\begin{align*}
x&\equiv a_1 \mod \alpha_1q_1,\\
x&\equiv a_2 \mod \alpha_2 q_2,\\
\vdots&\\
x&\equiv a_n\mod \alpha_n q_n,\end{align*}
where $\alpha_i=1$ if $2|p_i$ and $\alpha_i=2$ if $2\centernot |p_i$. Now if at most only one $p_i$ is divisible by $2$ then we can apply the Chinese Remainder Theorem to obtain the result. However, if $2 |p_{j}$ for all $j\in A\subset \{1,\dots,n\}$ where $|A|\geq 2$, then the Chinese Remainder Theorem is no longer immediately applicable since the moduli are not co-prime. Instead, we break the system down into co-prime factors, that is $x\equiv a_{j}\mod 2q_{j}$ implies $$x\equiv a_{j}\mod 2~~~~~{\rm and}~~~~~~x\equiv a_{j}\mod q_{j},$$ for all $j\in A$.  This new system can be solved using the Chinese Remainder Theorem, but only providing there are no inconsistencies with regards $x\equiv a_{j} \mod 2$; in particular, for all $j\in A$, $a_{j}$ must have the same parity. If there is an inconsistency, then we observe that due to the co-prime conditions on $q_i$ there is at most one $j_0\in A$ where $q_{j_0}$ is even, and we choose to make every other $a_{j}$ of the same parity. This is achieved by adding $q_{j}$ (which must be odd) to all those $a_{j}$ whose parity is different to $a_{j_0}$, and thus all the parities are now the same. This follows from the fact that $$q_{j}\times \frac{p_{j}\pi}{q_{j}}=p_{j}\pi,$$ which, since $2\centernot|p_j$, is an odd multiple of $\pi$ and so $$R(\theta_{j} (a_{j}+q_{j}))\vec{f}_{j}=R(\theta_{j} a_{j})\vec{f}_{j}+\pi.$$ Thus by eliminating any inconsistencies in the parity, we have only moved the vector $\vec{f}_j$ from its current quadrant into the diametrically opposite one. The Chinese Remainder Theorem can now be applied.

(Case Two) If $\theta_t=\frac{2\pi}{3}$ for some $t\in\{1,\dots,n\}$ then Lemma~\ref{5.62} is no longer valid for this particular quasi-momentum. However, the result still follows, since the two quadrants the vector $\vec{f}_t$ does visit under the action of the rotation matrix $R(\theta(\lambda_t))$ are indeed enough. This is because by Lemma~\ref{5.62} the vectors, $\vec{f}_i, i\in\{1,\dots,n\}\setminus\{t\}$, can be moved to one of these two quadrants under the corresponding action of $R(3\theta_i)$ (the new angle $3\theta_i$ ensuring $\vec{f}_t$ remains fixed in the correct quadrant whilst the other vectors are being moved about).~\qedhere
\end{proof}

We now state the main theorem of this section.

\begin{theorem}\label{5.64}
Let $J_T$ be an arbitrary period-$T$ Jacobi operator and $\{\lambda_i\}_{i=n+1}^\infty$ a sequence of complex numbers belonging to $\sigma_{ell}(J_T)$ where any collection $$\{\pi,\theta(\lambda_{n+1}),\theta(\lambda_{n+2}),\dots,\theta(\lambda_{n+k})\},~~k\geq 1,$$ is rationally independent, and $\{\lambda_1,\lambda_2,\dots,\lambda_n\}\subseteq\sigma_{ell}(J_T)$ where $$\theta(\lambda_1)=\frac{p_1\pi}{q_1},\theta(\lambda_2)=\frac{p_2\pi}{q_2},\dots,\theta(\lambda_n)=\frac{p_n\pi}{q_n},$$ $p_i,q_i\in\mathbb{N},\gcd(p_i,q_i)=1,\gcd(q_i,q_j)= 1$ for all $i\neq j$ and $\theta(\lambda_i)\neq\frac{\pi}{2}$ for all $i$. Then for any positive sequence $K(n)$ with $K(n)\rightarrow\infty$ arbitrarily slowly as $n\rightarrow\infty$, there exists a potential $|q_n|\leq \frac{K(n)}{n}$ for all $n$, and non-zero vectors $(u_{n,i})_{n\geq1}\in l^2(\mathbb{N};\mathbb{R})$ such that $$(J_T+Q)(u_{n,i})_{n\geq1}=\lambda_i (u_{n,i})_{n\geq1}$$ for all $i$ and where $Q$ is an infinite diagonal matrix with entries $(q_n)$.
\end{theorem}

\begin{proof}
The proof is the same as that used in Theorem~\ref{5.57}, except now when we divide up the eigenvalues into sets we have $\{\lambda_1,\lambda_2,\dots,\lambda_n\}\subseteq N_1$, and replace Lemma~\ref{5.43} with Lemma~\ref{5.63} to simultaneously rotate the vectors $\vec{f}_1,\dots,\vec{f}_n$ into shrinkable regions.~\qedhere
\end{proof}

\begin{remark}
Clearly, Theorem~\ref{5.64} covers all cases in Theorem~\ref{5.60}. However, Theorem~\ref{5.64} does not replace Theorems~\ref{5.73} or \ref{5.57} since the pair of eigenvalues with quasi-momenta rationally dependent with $\pi$ considered in these cases, say $\theta(\lambda_1)=\frac{p_1\pi}{q_1}, \theta(\lambda_2)=\frac{p_2\pi}{q_2}$, could be such that $q_1,q_2$ are not co-prime and for Theorem~\ref{5.64} we demand that all $q_i$ be pairwise co-prime.
\end{remark}

\section*{Acknowledgements}

The authors would like to thank K. M. Schmidt for useful comments that helped to improve this paper. E.J. was supported by the Engineering and Physical Sciences Research Council (grant EP/M506540/1). S.N. was supported by the Russian Science Foundation (project no.15-11-30007 ), RFBR  grant 16-11-00443a, NCN grant 2013/09/BST/04319 and LMS. S.N. also expresses his gratitude to the University of Kent at Canterbury for the support and hospitality.

\end{document}